\documentclass[aop,MSNbibl,citesort,seceqn,dvips]{arximspdf}

\doi{10.1214/11-AOP706}
\volume{41}
\issue{2}
\pubyear{2013}
\firstpage{961}
\lastpage{988}

\makeatletter
\newcommand{\eqref}[1]{(\ref{#1})}

\newlength{\myfboxsep}
\newlength{\mywidth}
\newcommand{\boxv}{\setlength{\myfboxsep}{\fboxsep}
\setlength{\fboxsep}{0pt}
\,\framebox[\mywidth]{$\vee$}\,
\setlength{\fboxsep}{\myfboxsep}}
\newtheorem{theorem}{Theorem}[section]
\newtheorem{lemma}{Lemma}[section]
\newtheorem{proposition}{Proposition}[section]
\newproclaim{definition}{Definition}[section]
\newproclaim{remark}{Remark}[section]
\newcommand{\prob}{\mathrm{P}}
\newcommand{\lito}{\mathrm{o}}
\newcommand{\bigo}{\mathrm{O}}

\newcommand{\gmu}{G_{\mu}}
\makeatother

\begin{document}
\begin{frontmatter}

\title{Free subexponentiality}
\runtitle{Free subexponentiality}

\begin{aug}
\author[A]{\fnms{Rajat~Subhra} \snm{Hazra}\ead[label=e1]{rajatmaths@gmail.com}}
\and
\author[A]{\fnms{Krishanu} \snm{Maulik}\corref{}\ead[label=e2]{krishanu@isical.ac.in}}
\runauthor{R. S. Hazra and K. Maulik}
\affiliation{Indian Statistical Institute}
\address[A]{
Statistics and Mathematics Unit\\
Indian Statistical Institute\\
203 B. T. Road\\
Kolkata 700108\\
India\\
\printead{e1}\\
\phantom{E-mail: }\printead*{e2}}
\end{aug}

\received{\smonth{9} \syear{2010}}
\revised{\smonth{8} \syear{2011}}

%
\begin{abstract}
In this article, we introduce the notion of \textit{free
subexponentiality}, which extends the notion of subexponentiality in
the classical probability setup to the noncommutative probability
spaces under freeness. We show that distributions with regularly
varying tails belong to the class of free subexponential distributions.
This also shows that the partial sums of free random elements having
distributions with regularly varying tails are tail equivalent to their
maximum in the sense of Ben~Arous and Voiculescu [\textit{Ann.
Probab.} \textbf{34} (2006) 2037--2059]. 
The analysis is based on the asymptotic relationship between the tail
of the distribution and the real and the imaginary parts of the
remainder terms in Laurent series expansion of Cauchy transform, as
well as the relationship between the remainder terms in Laurent series
expansions of Cauchy and Voiculescu transforms, when the distribution
has regularly varying tails.
\end{abstract}
%
%

\end{frontmatter}

\section{Introduction} \label{secintro}
A noncommutative probability space is a pair $(\mathcal{A},\tau)$
where~$\mathcal{A}$ is a unital complex algebra, and $\tau$ is a linear
functional on $\mathcal{A}$ satisfying $\tau(1)=1.$ A noncommutative
analog of independence, based on free products, was introduced by
Voiculescu \cite
{voiculescu1986addition}. A family of unital subalgebras $\{\mathcal
{A}_i\}_{i\in I}\subset\mathcal{A}$ is called \textit{free} if $\tau
(a_1\cdots a_n)=0$ whenever $\tau(a_j)=0, a_j\in\mathcal{A}_{i_j}$
and $i_j\neq i_{j+1}$ for all $j$. The above setup is suitable for
dealing with bounded random variables. In order to deal with unbounded
random variables, we need to consider a tracial $W^*$-probability space
$(\mathcal{A},\tau)$ with a von Neumann algebra $\mathcal{A}$ and a
normal faithful tracial state $\tau$.

A self-adjoint operator $X$ is said to be affiliated to a von Neumann
algebra~$\mathcal{A}$, if $f(X)\in\mathcal{A}$ for any bounded Borel
function $f$ on the real line $\mathbb{R}.$ A~self-adjoint operator
affiliated with $\mathcal{A}$ will also be called a random element. For
an affiliated random element (i.e., a self-adjoint operator) $X$, the
algebra generated by $X$ is defined as $\mathcal A_X = \{f(X): f\mbox{ bounded measurable}\}$. The notion of freeness was extended to
this context by Bercovici and Voiculescu~\cite{bercovici1993free}. A
set of random elements $\{
X_i\}_{1\leq i\leq k}$ affiliated with a von Neumann algebra $\mathcal
{A},$ are called freely independent, or simply free, if $\{\mathcal
A_{X_i}\}_{1\leq i\leq k}$ are free.

Given a random element $X$ affiliated with $\mathcal{A},$ the law of
$X$ is the unique probability measure $\mu_X$ on $\mathbb{R}$
satisfying $\tau(f(X))=\int_{-\infty}^{\infty}f(t)\,d\mu_X(t)$ for every
bounded Borel function $f$ on $\mathbb{R}.$ If $e_A$ denote the
projection valued spectral measure associated with $X$, evaluated at
the set $A$, then it is easy to see that $\mu_{X}(-\infty,x]=\tau
(e_{(-\infty,x]}(X)).$ The distribution function of $X$, denoted by
${F}_X$, is given by ${F}_{X}(x)=\mu_X(-\infty,x].$

Let $\mathcal{M}$ be the family of probability measures on $\mathbb
{R}$. On $\mathcal{M}$, two associative operations $*$ and $\boxplus$
can be defined. The measure $\mu*\nu$ is the classical convolution of
$\mu$ and $\nu$, which also corresponds to the probability law of a
random variable $X+Y$, where $X$ and $Y$ are independent and have laws
$\mu$ and $\nu$, respectively. Also, given two measures $\mu$ and
$\nu
$, there exists a unique measure $\mu\boxplus\nu$, called the
\textit
{free convolution} of $\mu$ and $\nu$, such that whenever $X$ and $Y$
are two free random elements on a tracial $W^*$ probability space
$(\mathcal A, \tau)$ with laws $\mu$ and $\nu$, respectively, $X+Y$ has
the law $\mu\boxplus\nu$. The free convolution was first introduced by
Voiculescu~\cite{voiculescu1986addition} for compactly supported measures,
extended by Maassen~\cite{maassen1992addition} to measures with finite
variance and by Bercovici and Voiculescu~\cite{bercovici1993free} to
arbitrary Borel
probability measures with unbounded support. The classical and free
convolutions of distributions are defined and denoted analogously.

The relationship between $*$ and $\boxplus$ convolution is very
striking. They have many similarities like characterizations of
infinitely divisible and stable laws~\cite{bercovici1999stable,bercovici2000free},
weak law of large numbers~\cite{bercovici1996law}
and central limit theorem~\cite{Voiculescu1985,maassen1992addition,pata1996central}.
Analogs of many other classical theories have also
been derived. In recent times, links with extreme value theory \cite
{arous2006free,arous2009free} and de Finetti-type theorems \cite
{banica907finetti} have drawn much attention in the literature.
However, there are differences too---for example, Cram\'er's
theorem~\cite{bercovici1995superconvergence} and Raikov's theorem~\cite
{benaych2004failure} fail in the noncommutative setup.

Now we consider an interesting family of distributions in the classical
setup called subexponential distributions. The main endeavor of this
article is to obtain an analog of this concept in the noncommutative
setup under freeness. A probability measure $\mu$ on $[0,\infty)$, with
$\mu(x,\infty)>0$ for all $x\geq0$, is said to be \textit
{subexponential}, if for every $n\in\mathbb{N}$,
\[
\mu^{*n}(x,\infty)\sim n\mu(x,\infty)\qquad \mbox{as $x\rightarrow
\infty$.}
\]
For a random variable $X$ with distribution $F$ and subexponential law
$\mu$, $X$ and $F$ are also called subexponential. The above definition
can be rephrased in terms of the complementary distribution functions.
For a distribution function $F$, we define its complementary
distribution function as $\overline F = 1 - F$. Then a subexponential
distribution function satisfies, for each natural number $n$,
$\overline
{F^{*n}}(x) \sim n \overline F(x)$ as $x\to\infty$. The definition can
be extended to probability measures $\mu$ and equivalently distribution
functions $F$ defined on the entire real line. A distribution function\vadjust{\goodbreak}
$F$ on the real line is called subexponential if the distribution
function $F_+$, defined as $F_+(x) = F(x)$, for $x\ge0$ and $F_+(x)=0$,
for $x<0$, is subexponential. Thus to discuss the subexponential
property of the probability measures, it is enough to consider the ones
concentrated on $[0,\infty)$. The subexponential random variables
satisfy \textit{the principle of one large jump} as well. If $\{X_i\}$
are i.i.d. subexponential random variables, then for all $n \in
\mathbb N$,
\[
\prob[X_1+\cdots+X_n>x]\sim n\prob[X_1>x]=\prob\Bigl[ {\max_{1\leq
i\leq n}}X_i>x\Bigr] \qquad\mbox{as $x\rightarrow\infty$.}
\]
Such a property makes subexponential distributions an ideal choice for
modeling ruin and insurance problems and has caused wide interest in
the classical probability literature; cf.
\cite{embrechtskluppelbergmikoschbook,rolskiteugelsbook}.

The classical definition of subexponential distributions can be easily
extended to the noncommutative setup by replacing the classical
convolution powers by free convolution powers. We shall define a free
subexponential measure on $[0, \infty)$ alone, but the definition can
be extended to probability measures on the entire real line, as in the
classical case. Formally, we define a free subexponential measure as follows:
\begin{definition}
A probability measure $\mu$ on $[0,\infty)$, with $\mu(x,\infty)>0$ for
all $x\geq0$, is said to be free subexponential if for all $n$,
\[
\mu^{\boxplus n}(x,\infty)=\underbrace{(\mu\boxplus\cdots\boxplus
\mu
)}_{n\ \mathrm{times}}(x,\infty)\sim n\mu(x,\infty)\qquad \mbox{as
$x\rightarrow\infty$.}
\]
\end{definition}

The above definition can be rewritten in terms of distribution
functions as well. A distribution function $F$ is called free
subexponetial if for all $n \in\mathbb N$, $\overline{F^{\boxplus
n}}(x) \sim n \overline F(x)$ as $x\to\infty$. A random variable $X$
affiliated to a tracial $W^*$-probability space is called free
subexponential if its distribution is so. One immediate consequence of
the definition of free subexponentiality is the principle of one large jump.

Ben Arous and Voiculescu~\cite{arous2006free} showed that for two
distribution functions $F$
and $G$, there exists a unique measure $F\boxv G$, such that whenever
$X$ and $Y$ are two free random elements on a tracial $W^*$-probability
space, $F\boxv G$ will become the distribution of $X\vee Y$. Here
$X\vee Y$ is the maximum of two self-adjoint operators defined using
the spectral calculus via the projection-valued operators; see \cite
{arous2006free} for details. Ben Arous and Voiculescu~\cite
{arous2006free} showed that $F\boxv
G(x) = \max((F(x)+G(x)-1),0)$, and hence $F^{\boxv n}(x) = \max
((nF(x)-(n-1)),0)$. Then we have for each $n$, $\overline{F^{\boxv n}}
(x) \sim n \overline F(x)$ as $x\to\infty$. Thus, by the definition of
free subexponentiality, we have
\begin{proposition}[(Free one large jump principle)] \label{propfreeonelargejump}
Free subexponential distributions satisfy the principle of one large
jump, namely, if $F$ is freely subexponential, then, for every $n$,
\[
\overline{F^{\boxplus n}}(x) \sim\overline{F^{\boxv n}}(x)\qquad \mbox{as
$x\to\infty$.}\vadjust{\goodbreak}
\]
\end{proposition}

While the class of free subexponential distributions possess the above
important property, it remains to be checked whether the class is
nonempty. The answer to this question, which is the main result of
this article, is given in Theorem~\ref{maintheorem-1}. The
distributions with regularly varying (right) tails of index $-\alpha$,
with $\alpha\ge0$, form an important class of examples of
subexponential distributions in the classical setup. (In further
discussions, we shall suppress the qualifier ``right.'') A (real
valued) measurable function $f$ defined on nonnegative real line is
called \textit{regularly varying} (at infinity) with index $\alpha$ if
for every $t>0$, $f(tx)/f(x)\rightarrow t^{\alpha}$ as $x\rightarrow
\infty$. If $\alpha=0$, then $f$ is said to be slowly varying (at
infinity). Regular variation with index $\alpha$ at zero is defined
analogously. In fact, $f$ is regularly varying at zero of index $\alpha
$, if the function $x\mapsto f(1/x)$ is regularly varying at infinity
of index $-\alpha$. Unless otherwise mentioned, the regular variation
of a function will be considered at infinity. For regular variation at
zero, we shall explicitly mention so. A distribution function $F$ on
$[0,\infty)$ has regularly varying tail of index $-\alpha$ if
$\overline F(x)$ is regularly varying of index $-\alpha$. Since
$\overline F(x) \to0$ as $x\to\infty$, we must necessarily have
$\alpha
\ge0$. As in the case of subexponential distributions, a distribution
$F$ on the entire real line is said to have regularly varying tail if
$F_+$ has so. Note that for $x>0$, we have $\overline{F_+}(x) =
\overline F(x)$. A probability measure with regularly varying tail is
defined through its distribution function. Equivalently, a measure $\mu
$ is said to have a regularly varying tail if $\mu(x,\infty)$ is
regularly varying.

Other than distributions with regularly varying tails, Weibull
distributions with shape parameter less than $1$ and lognormal
distribution are some other well-known examples of subexponential
distributions in the classical setup. The last two distributions have
all moments finite unlike the distributions with regularly varying
tails of index $-\alpha$, which have all moments higher than $\alpha$
infinite.

The distributions with regularly varying tails have already attracted
attention in the noncommutative probability theory. They play a very
crucial role in determining the domains of attraction of stable laws
\cite{bercovici1999stable,bercovici2000free}. In this article, we
shall show that the distributions with regularly varying tails form a
subclass of the free subexponential distributions.
\begin{theorem}\label{maintheorem-1}
If a distribution function $F$ has regularly varying tail of index
${-\alpha}$ with $\alpha\ge0$, then $F$ is free subexponential.
\end{theorem}

The class of distribution functions with regularly varying tails is a
significantly large class containing stable distributions, Pareto and
Fr\'echet distributions. For all $\alpha\ge0$, there are distribution
functions, which have regularly varying tail of index~$-\alpha$. The
class of distribution function with regularly varying tail of index
$-\alpha$ have found significant application in finance, insurance,
weather, Internet traffic modeling and many other fields.

While it need not be assumed that the measure is concentrated on
$[0,\infty)$, both the notions of free subexponentiality and regular
variation are defined in terms of the measure restricted to $[0,\infty
)$. Thus we shall assume the measure to be supported on $[0,\infty)$,
except for the definitions of the relevant transforms in the initial
part of Section~\ref{subsectransf} and in the statement and the
proof of Theorem~\ref{maintheorem-1}. Due to the lack of coordinate
systems and expressions for joint distributions of noncommutative
random elements in terms of probability measures, the proofs of the
above results deviate from the classical ones. In absence of the higher
moments of the distributions with regularly varying tails, we cannot
use the usual moment-based approach used in free probability theory.
Instead, Cauchy and Voiculescu transforms become the natural tools to
deal with the free convolution of measures. We recall the notions of
these transforms in Section~\ref{sectransform}. We then discuss the
relationship between the remainder terms of Laurent expansions of
Cauchy and Voiculescu transforms of measures with regularly varying
tail of index $-\alpha$. We need to consider four cases separately
depending on the maximum number $p$ of integer moments that the measure
$\mu$ may have. For a nonnegative integer $p$, let us denote the class
of all probability measures $\mu$ on $[0,\infty)$ with $\int
_0^\infty
t^p \,d\mu(t) <\infty$, but $\int_0^\infty t^{p+1} \,d\mu(t) =\infty$, by
$\mathcal M_p$. We shall also denote the class of all probability
measures $\mu$ in $\mathcal M_p$ with regularly varying tail of index
$-\alpha$ by $\mathcal M_{p,\alpha}$. Note that we necessarily have
$\alpha\in[p,p+1]$; cf. \cite[Proposition A3.8(d)]{embrechtskluppelbergmikoschbook}.
Theorems \ref
{thmerrorequiv}--\ref{thmerrorequivnew} summarize the
relationships among the remainder terms for various choices of $\alpha$
and $p$. These theorems are the key tools of this article. Section \ref
{sectransform} is concluded with two Abel--Tauber-type results for
Stieltjes transform of measures with regularly varying tail. We then
prove Theorem~\ref{maintheorem-1} in Section~\ref{secpfthmone}
using Theorems~\ref{thmerrorequiv}--\ref{thmerrorequivnew}. We
use the final two sections to prove Theorems~\ref{thmerrorequiv}--\ref{thmerrorequivnew}. In Section~\ref{secCauchy}, we
collect some results about the remainder term in Laurent series
expansion of Cauchy transform of measures with regularly varying tails.
In Section~\ref{secC-Vreln}, we study the relationship between the
remainder terms in Laurent expansions of Cauchy and Voiculescu
transforms through a general analysis of the remainder terms of Taylor
expansions of a suitable class of functions and their inverses or
reciprocals. Combining the results of Sections~\ref{secCauchy}
and~\ref{secC-Vreln}, we prove Theorems~\ref{thmerrorequiv}--\ref
{thmerrorequivnew}.

\section{Some transforms and their related properties} \label{sectransform}
In this section, we collect some notation, definitions and results to
be used later in the article. In Section~\ref{subsecnontang}, we
define the concept of nontangential limits. Various transforms in
noncommutative probability theory, like Cauchy, Voiculescu and $R$
transforms are introduced in Section~\ref{subsectransf}.
Theorems~\ref{thmerrorequiv}--\ref{thmerrorequivnew} regarding
the relationship between the remainder terms of Laurent expansions of
Cauchy and Voiculescu transforms are given in this subsection as well.
Finally, in Section~\ref{subseckaramata}, two results about
measures with regularly varying tails are given.
\subsection{Nontangential limits and notation} \label{subsecnontang}
The complex plane will be denoted by $\mathbb C$ and for a complex
number $z$, $\Re z$ and $\Im z$ will denote its real and imaginary
parts respectively. We say $z$ goes to infinity (zero, resp.) \textit
{nontangentially} to $\mathbb R$ (n.t.), if $z$ goes to infinity
(zero, resp.), while $\Re z/ \Im z$ stays bounded. We can then define
that a
function $f$ converges or stays bounded as $z$ goes to infinity (or
zero) n.t. To elaborate upon the notion, given positive numbers $\eta$,
$\delta$ and $M$, let us define the following cones:
\begin{enumerate}[(1)]
\item[(1)] $\Gamma_{\eta}=\{z\in\mathbb{C}^+: |\Re z|<\eta\Im z\}$
and $\Gamma_{\eta,M}=\{z\in\Gamma_{\eta}: |z|>M\}$;
\item[(2)] $\Delta_{\eta}=\{z\in\mathbb{C}^-: |\Re z|<-\eta\Im z\}$
and $\Delta_{\eta,\delta}=\{z\in\Delta_{\eta}: |z|<\delta\}$,
\end{enumerate}
where $\mathbb C^+$ and $\mathbb C^-$ are the upper and the lower
halves of the complex plane respectively, namely, $\mathbb C^+ = \{z\in
\mathbb C: \Im z>0\}$ and $\mathbb C^- = - \mathbb C^+$. Then we shall
say that $f(z) \to l$ as $z$ goes to $\infty$ n.t., if for any
$\varepsilon
>0$ and $\eta>0$, there exists $M\equiv M(\eta,\varepsilon)>0$, such that
$|f(z) - l|< \varepsilon$, whenever $z \in\Gamma_{\eta, M}$. The
boundedness can be defined analogously.

We shall write $f(z)\approx g(z)$, $f(z)=\lito(g(z))$ and $f(z)=\bigo
(g(z))$ as $z\to\infty$ n.t. to mean that $f(z)/g(z)$ converges to a
nonzero limit, ${f(z)}/{g(z)}\rightarrow0$ and $f(z)/g(z)$ stays
bounded as $z\to\infty$ n.t., respectively. If the nonzero limit is
$1$ in the first case, we write $f(z) \sim g(z)$ as $z\to\infty$ n.t.
For $f(z)=\lito(g(z))$ as $z\to\infty$ n.t., we shall also use the
notation $f(z)\ll g(z)$ and $g(z)\gg f(z)$ as $z\to\infty$ n.t.

The map $z\mapsto1/z$ maps the set $\Gamma_{\eta,{1}/{\delta}}$ onto
$\Delta_{\eta,\delta}$ for each positive $\eta$ and $\delta$. Thus the
analogous concepts can be defined for $z\to0$ n.t. using $\Delta
_{\eta
,\delta}$.

\subsection{Cauchy and Voiculescu transforms} \label{subsectransf}
For a probability measure $\mu\in\mathcal{M}$, its Cauchy transform is
defined as
\[
G_{\mu}(z)=\int_{-\infty}^{\infty}\frac{1}{z-t}\,d\mu(t),\qquad z\in
\mathbb{C}^+.
\]
Note that $G_{\mu}$ maps $\mathbb{C}^+$ to $\mathbb{C}^{-}$. Set
$F_{\mu
}=1/\gmu$, which maps $\mathbb{C}^+$ to $\mathbb{C}^+$. We shall be
also interested in the function $H_\mu(z) = G_\mu(1/z)$ which maps
$\mathbb C^-$ to~$\mathbb C^-$.

By Proposition 5.4 and Corollary 5.5 of~\cite{bercovici1993free}, for
all $\eta>0$ and for all $\varepsilon\in(0,\eta\wedge1)$, there exists
$\delta\equiv\delta(\eta)$ small enough, such that $H_\mu$ is a
conformal bijection from $\Delta_{\eta,\delta}$ onto an open set
$\mathcal{D}_{\eta,\delta}$, where the range sets satisfy
\[
\Delta_{\eta-\varepsilon,(1-\varepsilon)\delta}\subset\mathcal
{D}_{\eta
,\delta}\subset\Delta_{\eta+\varepsilon,(1+\varepsilon)\delta}.
\]

If we define $\mathcal D = \bigcup_{\eta>0} \mathcal D_{\eta, \delta
(\eta
)}$, then we can obtain an analytic function $L_\mu$ with domain
$\mathcal D$ by patching up the inverses of $H_\mu$ on $\mathcal
{D}_{\eta,\delta(\eta)}$ for each $\eta>0$. In this case $L_{\mu}$
becomes the right inverse of $H_\mu$ on $\mathcal{D}.$ Also it was
shown that the sets of type $\Delta_{\eta,\delta}$ were contained in
the unique connected component of the set $H_\mu^{-1}(\mathcal{D})$. It
follows that $H_\mu$ is the right inverse of $L_{\mu}$ on $\Delta
_{\eta
,\delta}$ and hence on the whole connected component by analytic continuation.

We then define $R$ and Voiculescu transforms of the probability measure
$\mu$ respectively as
%
%
\begin{equation}\label{eqRphidefn}
R_{\mu}(z)=\frac{1}{L_{\mu}(z)}-\frac{1}{z}\quad \mbox{and}\quad \phi
_{\mu}(z)=R_{\mu}({1}/{z}).
\end{equation}
Arguing as in the case of $\gmu(1/z)$, it can be shown that $F_\mu$ has
a left inverse, denoted by $F_\mu^{-1}$ on a suitable domain and, in
that case, we have
\[
\phi_{\mu}(z)=F_{\mu}^{-1}(z)-z.
\]

Bercovici and Voiculescu~\cite{bercovici1993free} established the
following relation between
free convolution and Voiculescu and $R$ transforms. For probability
measures $\mu$ and $\nu$,
\[
\phi_{\mu\boxplus\nu}=\phi_{\mu}+\phi_{\nu}\quad \mbox{and}\quad R_{\mu
\boxplus\nu}=R_{\mu}+R_{\nu},
\]
wherever all the functions involved are defined.

We shall also need to analyze the power and Taylor series expansions of
the above transforms. For Taylor series expansion of a function, we
need to define the remainder term appropriately, so that it becomes
amenable to the later calculations. In fact, for a function $A$ with
Taylor series expansion of order $p$, we define the remainder term as
%
%
\begin{equation} \label{eqdefremainder}
r_A(z) = z^{-p} \Biggl( A(z) - \sum_{i=0}^p a_i z^i \Biggr).
\end{equation}
Note that we divide by $z^p$ after subtracting the polynomial part.

For compactly supported measure $\mu$, Speicher \cite
{speicher1994multiplicative} showed that, in an appropriate
neighborhood of zero, $R_{\mu}(z)=\sum_{j= 0}^\infty\kappa
_{j+1}(\mu
)z^j$, where $\{\kappa_j(\mu)\}$ denotes the free cumulant sequence of
the probability measure $\mu$. For probability measures $\mu$ with
finite $p$ moments, Taylor expansions of $R_\mu$ and $H_\mu$ are given
by Theorems~$1.3$ and $1.5$ of~\cite{benaych2006taylor},
%
%
\begin{eqnarray} \label{eqerror}
R_{\mu}(z)&=&\sum_{j=0}^{p-1}\kappa_{j+1}(\mu)z^j+ z^{p-1} r_{R_\mu}(z)
\quad\mbox{and}
\nonumber
\\[-8pt]
\\[-8pt]
\nonumber
 H_\mu(z)&=&\sum_{j=1}^{p+1}m_{j-1}(\mu)z^j+z^{p+1}
r_{H_\mu}(z),
\end{eqnarray}
where the remainder terms $r_{R_\mu}(z) \equiv r_R(z) = \lito(1)$ and
$r_{H_\mu}(z) \equiv r_H(z) = \lito(1)$ as $z\rightarrow0$ n.t. are
defined along the lines of \eqref{eqdefremainder}, $\{\kappa_j(\mu
):j\le p\}$ denotes the free cumulant sequence of $\mu$ as before and
$\{m_j(\mu):j\le p\}$ denotes the moment sequence of the probability
measure $\mu$. When there is no possibility of confusion, we shall
sometimes suppress the measure involved in the notation for the moment
and the cumulant sequences, as well as the remainder terms.
In the study of stable laws and the infinitely divisible laws, the
following relationship between\vadjust{\goodbreak} Cauchy and Voiculescu transforms of a
probability measure $\mu$, obtained in Proposition 2.5 of Bercovici
and Pata \cite
{bercovici1999stable}, played a crucial role:
%
%
\begin{equation} \label{eqphimu}
\phi_{\mu}(z)\sim z^2\biggl[\gmu(z)-\frac{1}{z}\biggr]\qquad \mbox{as }
z\rightarrow\infty\mbox{ n.t.}
\end{equation}
Depending on the number of moments that the probability measure $\mu$
may have, its Cauchy and Voiculescu transforms can have Laurent series
expansions of higher order. Motivated by this fact, for probability
measures $\mu\in\mathcal M_p$ (i.e., when $\mu$ has only $p$
integral moments), we introduce the remainder terms in Laurent series
expansion of Cauchy and Voiculescu transforms (in analogy to the
remainder terms in Taylor series expansion),
%
\begin{equation}\label{eqrgdefn}
r_{G_\mu}(z) \equiv r_G(z) = z^{p+1} \Biggl(G_\mu(z) - \sum
_{j=1}^{p+1}m_{j-1}(\mu)z^{-j}\Biggr)
\end{equation}
and
\begin{equation}\label{eqrphidefn}
r_{\phi_\mu}(z) \equiv r_\phi(z) = z^{p-1} \Biggl(\phi_\mu(z) - \sum
_{j=0}^{p-1}\kappa_{j+1}(\mu)z^{-j}\Biggr),
\end{equation}
where we shall again suppress the measure $\mu$ in the notation if
there is no possibility of confusion. In \eqref{eqrphidefn}, we
interpret the sum on the right-hand side as zero, when $p=0$. Using the
remainder terms defined in \eqref{eqrgdefn} and \eqref{eqrphidefn}
we provide extensions of \eqref{eqphimu} in Theorems~\ref{thmerrorequiv}--\ref{thmerrorequivnew} for different choices of $\alpha$
and $p$. We split the statements into four cases as follows: (i) $p$ is
a positive integer, and $\alpha\in(p,p+1)$; (ii)~$p$ is a positive
integer, and $\alpha=p$; (iii) $p=0$, and $\alpha\in[0,1)$; (iv) $p$ is
a nonnegative integer and, $\alpha=p+1$, giving rise to Theorems \ref
{thmerrorequiv}--\ref{thmerrorequivnew}, respectively.

We first consider the case where $p$ is a positive integer and $\alpha
\in(p,p+1)$.

\begin{theorem}\label{thmerrorequiv}
Let $\mu$ be a probability measure in the class $\mathcal M_p$ and
$\alpha\in(p,p+1)$. The following statements are equivalent:
\begin{enumerate}[(iii)]
\item[(i)]\hypertarget{tail}$\mu(y,\infty)$ is regularly varying of index $-\alpha$.
\item[(ii)]\hypertarget{Cauchyremainder}$\Im r_G(iy)$ is regularly varying of index $-(\alpha-p)$.
\item[(iii)]\hypertarget{Voiculescuremainder} $\Im r_\phi(iy)$ is regularly varying of index $-(\alpha-p)$,
$\Re r_\phi(iy)\gg y^{-1}$ as $y\to\infty$ and $r_\phi(z)\gg
z^{-1}$ as
$z\to\infty$ n.t.
\end{enumerate}
If any of the above statements holds, we also have, as $z\to\infty$ n.t.,
%
%
\begin{equation} \label{rg-rphi}
r_G(z)\sim r_\phi(z) \gg z^{-1};
\end{equation}
as $y\to\infty$,
%
%
\begin{equation}
\Im r_\phi(iy) \sim\Im r_G(iy) \sim-\frac{\pi(p+1-\alpha
)/2}{\cos({\pi(\alpha-p)}/2)} y^p \mu(y,\infty) \gg\frac1y
\label{imrg-rphi}
\end{equation}
and
%
%
\begin{equation}
\Re r_\phi(iy) \sim\Re r_G(iy) \sim-\frac{{\pi(p+2-\alpha
)}/2}{\sin({\pi(\alpha-p)}/2)} y^p \mu(y,\infty) \gg\frac1y.
\label{rerg-rphi}
\end{equation}
\end{theorem}

Next we consider the case where $p$ is a positive integer and $\alpha=p$.
\begin{theorem}\label{thmerrorequiveqp}
Let $\mu$ be a probability measure in the class $\mathcal M_p$. The
following statements are equivalent:
\begin{enumerate}[(iii)]
\item[(i)]$\mu(y,\infty)$ is regularly varying of index $-p$. \label{taileqp}
\item[(ii)]$\Im r_G(iy)$ is slowly varying. \label{Cauchyremaindereqp}
\item[(iii)]$\Im r_\phi(iy)$ is slowly varying, $\Re r_\phi(iy)\gg y^{-1}$ as
$y\to\infty$ and $r_\phi(z)\gg z^{-1}$ as $z\to\infty$ n.t. \label
{Voiculescuremaindereqp}
\end{enumerate}
If any of the above statements holds, we also have, as $z\to\infty$ n.t.,
%
%
\begin{equation} \label{rg-rphieqp}
r_G(z)\sim r_\phi(z) \gg z^{-1};
\end{equation}
as $y\to\infty$,
%
%
\begin{equation}
\Im r_\phi(iy) \sim\Im r_G(iy) \sim-\frac\pi2 y^p \mu(y,\infty)
\gg
\frac1y \label{imrg-rphieqp}
\end{equation}
and
%
%
\begin{equation}
\Re r_\phi(iy) \sim\Re r_G(iy) \gg\frac1y. \label{rerg-rphieqp}
\end{equation}
\end{theorem}

In the third case, we consider $\alpha\in[0,1)$.
\begin{theorem}\label{thmerrorequiv-0}
Let $\mu$ be a probability measure in the class $\mathcal M_0$ and
$\alpha\in[0,1)$. The following statements are equivalent:
\begin{enumerate}[(iii)]
\item[(i)]$\mu(y,\infty)$ is regularly varying of index $-\alpha$.
\label{tail-0}
\item[(ii)]$\Im r_G(iy)$ is regularly varying of index $-\alpha$. \label
{Cauchyremainder-0}
\item[(iii)]\hypertarget{Voiculescuremainder-0}$\Im r_\phi(iy)$ is regularly varying of index $-\alpha$, $\Re
r_\phi(iy)\approx\Im r_\phi(iy)$ as $y\to\infty$ and $r_\phi
(z)\gg
z^{-1}$ as $z\to\infty$ n.t.
\end{enumerate}
If any of the above statements holds, we also have, as $z\to\infty$ n.t.,
%
%
\begin{equation} \label{rg-rphi-0}
r_G(z)\sim r_\phi(z) \gg z^{-1};
\end{equation}
as $y\to\infty$,
%
%
\begin{equation}
\Im r_\phi(iy) \sim\Im r_G(iy) \sim-\frac{{\pi(1-\alpha
)}/2}{\cos({\pi\alpha}/2)} \mu(y,\infty) \gg\frac1y \label{imrg-rphi-0}
\end{equation}
and
%
%
\begin{equation}
\Re r_\phi(iy) \sim\Re r_G(iy) \sim-d_\alpha\mu(y,\infty) \gg
\frac
1y, \label{rerg-rphi-0}
\end{equation}
where
\[
d_\alpha=
\cases{
\displaystyle\frac{{\pi(2-\alpha)}/2}{\sin({\pi\alpha}/2)}, &\quad $\mbox{when
$\alpha>0$,}$\vspace*{2pt}\cr
1, &\quad $\mbox{when $\alpha=0$.}$}
\]
\end{theorem}

Finally, we consider the case where $p$ is a nonnegative integer, and
$\alpha=p+1$.
\begin{theorem}\label{thmerrorequivnew}
Let $\mu$ be a probability measure in the class $\mathcal M_p$ and
$\beta\in(0,1/2)$. The following statements are equivalent:
\begin{enumerate}[(iii)]
\item[(i)]$\mu(y,\infty)$ is regularly varying of index $-(p+1)$. \label
{tailnew}
\item[(ii)]$\Re r_G(iy)$ is regularly varying of index $-1$. \label{Cauchyremaindernew}
\item[(iii)]$\Re r_\phi(iy)$ is regularly varying of index $-1$, $y^{-1} \ll
\Im r_\phi(iy)\ll y^{-(1-\beta/2)}$ as $y\to\infty$ and $z^{-1} \ll
r_\phi(z) \ll z^{-\beta}$ as $z\to\infty$ n.t. \label{Voiculescuremaindernew}
\end{enumerate}
If any of the above statements holds, we also have, as $z\to\infty$ n.t.,
%
%
\begin{equation} \label{rg-rphinew}
z^{-1}\ll r_G(z)\sim r_\phi(z)\ll z^{-\beta};
\end{equation}
as $y\to\infty$,
%
%
\begin{equation} \label{rerg-rphinew}
y^{-(1+\beta/2)} \ll\Re r_\phi(iy) \sim\Re r_G(iy) \sim-\frac\pi2
y^p \mu(y,\infty) \ll y^{-(1-\beta/2)}
\end{equation}
and
%
%
\begin{equation} \label{imrg-rphinew}
y^{-1} \ll\Im r_\phi(iy) \sim\Im r_G(iy) \ll y^{-(1-\beta/2)}.
\end{equation}
\end{theorem}

It is easy to obtain the equivalent statements for $H_\mu$ and $R_\mu$
through the simple observation that $G_\mu(z)=H_\mu(1/z)$ and $\phi
_\mu
(z)=R_\mu(1/z)$. For $p=0$, Theorems~\ref{thmerrorequiv-0}
and \ref
{thmerrorequivnew} together give a special case of \eqref{eqphimu} for the probability measures with regularly varying tail and
infinite mean. However, Theorems~\ref{thmerrorequiv}--\ref{thmerrorequivnew} give more detailed asymptotic behavior of the real and
imaginary parts separately, which is required for our analysis.

\subsection{Karamata-type results} \label{subseckaramata}
We provide here two results for regularly varying functions, which we
shall be using in the proofs of our results. They are variants of
Karamata's Abel--Tauber theorem for Stieltjes transform (cf.~\cite{binghamgoldieteugels1987}, Section 1.7.5) and explain the
regular variation of Cauchy transform of measures with regularly
varying tails.

The first result is quoted from~\cite{bercovici1999stable}.
\begin{proposition}[(\cite{bercovici1999stable}, Corollary 5.4)]\label{stieltjes}
Let $\rho$ be a positive Borel measure on $[0,\infty)$ and fix
$\alpha
\in[0,2)$. Then the following statements are equivalent:
\begin{enumerate}[(ii)]
\item[(i)]$y\mapsto\rho[0,y]$ is regularly varying of index $\alpha$.
\item[(ii)]$y\mapsto\int_0^{\infty}\frac{1}{t^2+y^2}\,d\rho(t)$ is regularly
varying of index $-(2-\alpha)$.\vadjust{\goodbreak}
\end{enumerate}
If either of the above conditions is satisfied, then
\[
\int_0^{\infty}\frac{1}{t^2+y^2}\,d\rho(t)\sim\frac{{\pi
\alpha
}/{2}}{\sin({\pi\alpha}/{2})}\frac{\rho[0,y]}{y^2}\qquad \mbox{as }
y\rightarrow\infty.
\]
The constant pre-factor on the right-hand side is interpreted as $1$
when $\alpha=0$.
\end{proposition}

The second result uses a different integrand.
\begin{proposition}\label{stieltjes2}
Let $\rho$ be a finite positive Borel measure on $[0,\infty)$ and fix
$\alpha\in[0,2)$. Then the following statements are equivalent:
\begin{enumerate}[(ii)]
\item[(i)]$y\mapsto\rho(y,\infty)$ is regularly varying of index
$-\alpha$.
\item[(ii)]$y\mapsto\int_0^{\infty}\frac{t^2}{t^2+y^2}\,d\rho(t)$ is regularly
varying of index $-\alpha$.
\end{enumerate}
If either of the above conditions is satisfied, then
\[
\int_0^{\infty}\frac{t^2}{t^2+y^2}\,d\rho(t)\sim\frac{{\pi
\alpha
}/{2}}{\sin({\pi\alpha}/{2})}{\rho(y,\infty)}\qquad \mbox{as }
y\rightarrow
\infty.
\]
The constant pre-factor on the right-hand side is interpreted as $1$
when $\alpha=0$.
\end{proposition}
\begin{pf}
Define $d\widetilde\rho(s)=\rho(\sqrt s,\infty)\,ds$. By a variant of
Karamata's theorem, given in Theorem 0.6(a) of~\cite{resnickbook}, as
$\alpha<2$, we have
%
%
\begin{equation}\label{eqstieltjesrho}
\widetilde\rho[0,y]\sim\frac1{1-\alpha/2} y \rho\bigl(\sqrt
{y},\infty\bigr)
\end{equation}
is regularly varying of index $1-\alpha/2$. Then we have
\begin{eqnarray*}
\int_0^\infty\frac{t^2}{t^2+y^2} \,d\rho(t)& =&y^2 \int_0^\infty\int_0^t
\frac{2sds}{(s^2+y^2)^2} \,d\rho(t)\\
& =& y^2 \int_0^\infty\frac{2s
\rho
(s,\infty)}{(s^2+y^2)^2} \,ds = y^2 \int_0^\infty\frac{d\widetilde
\rho
(s)}{(s+y^2)^2}.
\end{eqnarray*}
Now, first applying Theorem 1.7.4 of~\cite{binghamgoldieteugels1987}
as $\widetilde\rho[0,y]$ is regularly varying of index $1-\alpha
/2\in
(0,2]$ and then \eqref{eqstieltjesrho}, we have
\[
\int_0^\infty\frac{t^2}{t^2+y^2} \,d\rho(t) \sim\frac{(1-
\alpha/2) {\pi\alpha}/{2}}{\sin({\pi\alpha}/{2})} y^2\frac
{\widetilde\rho[0,y^2]}{y^4} \sim\frac{{\pi\alpha}/{2}}{\sin
({\pi\alpha}/{2})} \rho(y,\infty).
\]
\upqed\end{pf}

\section{Free subexponentiality of measures with regularly varying
tails} \label{secpfthmone}
We now use Theorems~\ref{thmerrorequiv}--\ref{thmerrorequivnew}
to prove Theorem~\ref{maintheorem-1}. We shall first look at the tail
behavior of the free convolution of two probability measures with
regularly varying tails and which are tail balanced. Theorem~\ref{maintheorem-1} will be proved by suitable choices of the two measures.
\begin{lemma} \label{lemfreetail}
Suppose $\mu$ and $\nu$ are two probability measures on $[0,\infty)$
with regularly varying tails, which are tail balanced; that is, for\vadjust{\goodbreak}
some $c>0$, we have $\nu(y,\infty)\sim c \mu(y,\infty)$. Then
\[
\mu\boxplus\nu(y,\infty) \sim(1+c) \mu(y,\infty).
\]
\end{lemma}
\begin{pf}
We shall now indicate the associated probability measures in the
remainder terms, moments and the cumulants to avoid any confusion.
Since $\mu$ and~$\nu$ are tail balanced and have regularly varying
tails, for some nonnegative integer $p$ and $\alpha\in[p,p+1]$, we
have both $\mu$ and $\nu$ in the same class $\mathcal M_{p,\alpha}$.
When $\alpha\in[p,p+1)$, depending on the choice of $p$ and $\alpha$,
we apply one of Theorems~\ref{thmerrorequiv},~\ref{thmerrorequiveqp} and~\ref{thmerrorequiv-0} on the imaginary parts of the
remainder terms in Laurent expansion of Voiculescu transforms. On the
other hand, for $\alpha=p+1$, we apply Theorem~\ref{thmerrorequivnew} on the real parts of the corresponding objects. We work out only
the case $\alpha\in[p,p+1)$ in details, while the other case $\alpha
=p+1$ is similar.

For $\alpha\in[p, p+1)$, by Theorems~\ref{thmerrorequiv}--\ref{thmerrorequiv-0}, we have
%
\begin{eqnarray}
r_{\phi_\mu}(z) &\gg& z^{-1} \quad\mbox{and} \quad r_{\phi_\nu}(z) \gg
z^{-1},\label{Voi}\\
\Re r_{\phi_\mu}(-iy) &\gg& y^{-1}\quad \mbox{and}\quad \Re r_{\phi_\nu
}(-iy) \gg y^{-1},\label{realVoi}\\
\label{imagVoi}\Im r_{\phi_\mu}(iy)& \sim& -\frac{{\pi(p+1-\alpha)}/2}{\cos
({\pi
(\alpha-p)}/2)} y^{p} \mu(y,\infty) \quad\mbox{and}
\nonumber
\\[-8pt]
\\[-8pt]
\nonumber
 \Im r_{\phi
_\nu}(iy) &\sim& -\frac{{\pi(p+1-\alpha)}/2}{\cos({\pi
(\alpha
-p)}/2)} y^{p} \nu(y,\infty).
\end{eqnarray}
For $p=0$ and $\alpha\in[0,1)$, we further have
\begin{eqnarray}\label{balance}
\Im r_{\phi_\mu}(iy) &\approx&\Re r_{\phi_\mu}(iy) \approx\mu
(y,\infty)
\quad\mbox{and}
\nonumber
\\[-8pt]
\\[-8pt]
\nonumber
\Im r_{\phi_\nu}(iy) &\approx&\Re r_{\phi_\nu
}(iy) \approx\nu(y,\infty).
\end{eqnarray}

We also know that both Voiculescu transforms and cumulants add up in
case of free convolution. Hence,
%
%
\begin{equation}\label{sumVoi}
r_{\phi_{\mu\boxplus\nu}}(z) = r_{\phi_\mu}(z) + r_{\phi_\nu}(z).
\end{equation}
Further, we shall have
$\kappa_p(\mu\boxplus\nu)<\infty$, but $\kappa_{p+1}(\mu\boxplus
\nu
)=\infty$ and similar results hold for the moments of $\mu\boxplus
\nu$
as well. Then Theorems~\ref{thmerrorequiv}--\ref{thmerrorequiv-0}
will also apply for $\mu\boxplus\nu$. Thus, applying \eqref{sumVoi}
and its real and imaginary parts evaluated at $z=iy$, together
with \eqref{Voi}--\eqref{balance}, respectively, we get
\begin{eqnarray*}
r_{\phi_{\mu\boxplus\nu}}(z) &\gg &z^{-1} \qquad\mbox{as $z\to\infty$ n.t.},
\\
\Re r_{\phi_{\mu\boxplus\nu}}(iy) &\gg& y^{-1} \qquad\mbox{as $y\to
\infty$}
\end{eqnarray*}
and
%
%
\begin{equation}
\qquad\Im r_{\phi_{\mu\boxplus\nu}}(iy) \sim-(1+c) \frac{{\pi
(p+1-\alpha
)}/2}{\cos({\pi(\alpha-p)}/2)} y^{p} \mu(y,\infty) \qquad\mbox{as
$y\to
\infty$}, \label{eqasymp1}
\end{equation}
which is regularly varying of index $-(\alpha-p)$. Further, for $p=0$
and $\alpha\in[0,1)$, we have
\[
\Im r_{\phi_{\mu\boxplus\nu}}(iy) \approx\Re r_{\phi_{\mu
\boxplus\nu}}(iy).
\]
In the last two steps, we also use the hypothesis that $\nu(y,\infty)
\sim c \mu(y,\infty)$ as $y\to\infty$. Thus, again using
Theorems \ref
{thmerrorequiv}--\ref{thmerrorequiv-0}, we have
%
%
\begin{equation} \label{eqasymp2}
-\frac{{\pi(p+1-\alpha)}/2}{\cos({\pi(\alpha-p)}/2)}
y^{p} \mu
\boxplus\nu(y, \infty) \sim\Im r_{\phi_{\mu\boxplus\nu}}(iy).
\end{equation}
Combining \eqref{eqasymp1} and \eqref{eqasymp2}, the result follows.
\end{pf}

We are now ready to prove the subexponentiality of a distribution with
regularly varying tail.
\begin{pf*}{Proof of Theorem \protect\ref{maintheorem-1}}
Let $\mu$ be the probability measure on $[0,\infty)$ associated with
the distribution function $F_+$. Then $\mu$ also has regularly varying
tail of index $-\alpha$. We prove that
%
%
\begin{equation} \label{eqindconv}
\mu^{\boxplus n}(y,\infty) \sim n \mu(y,\infty)\qquad \mbox{as $y\to
\infty$}
\end{equation}
by induction on $n$. To prove \eqref{eqindconv} for $n=2$, apply
Lemma~\ref{lemfreetail} with both the probability measures as $\mu$
and the constant $c=1$. Next assume \eqref{eqindconv} holds for
$n=m$. To prove \eqref{eqindconv} for $n=m+1$, apply Lemma~\ref{lemfreetail} again with the probability measures~$\mu$ and $\mu
^{\boxplus
m}$ and the constant $c=m$.
\end{pf*}

\section{Cauchy transform of measures with regularly varying tail}
\label{secCauchy}
As a first step toward proving Theorems~\ref{thmerrorequiv}--\ref
{thmerrorequivnew}, we now collect some results about $r_G(z)$,
when the probability measure $\mu$ has regularly varying tails. These
results will be be useful in showing equivalence between the tail of
$\mu$ and $r_G(iy)$. It is easy to see by induction that
\[
\frac1{z-t} - \sum_{j=0}^p \frac{t^j}{z^{j+1}} = \biggl(\frac tz
\biggr)^{p+1} \frac1{z-t}.
\]
Integrating and multiplying by $z^{p+1}$, we get
%
%
\begin{equation}\label{eqrG}
r_G(z) = \int_0^\infty\frac{t^{p+1}}{z-t} \,d\mu(t).
\end{equation}
We use \eqref{eqrG} to obtain asymptotic upper and lower bounds for
$r_G(z)$ as $z\to\infty$ n.t. Similar results about $r_H$ can be
obtained easily from the fact that $r_G(z)=r_H(1/z)$, but will not be
stated separately. We consider the lower bound first.

\begin{proposition}\label{proplowerbdrG}
Suppose $\mu\in\mathcal M_p$ for some nonnegative integer $p$, then
\[
z^{-1}\ll r_G(z) \qquad\mbox{as $z\to\infty$ n.t.}\vadjust{\goodbreak}
\]
\end{proposition}
\begin{pf}
We need to show that for any $\eta>0$, as $|z|\to\infty$ with $z$ in
the cone $\Gamma_\eta$, we have $|zr_G(z)|\to\infty$. Note that for
$z=x+iy\in\Gamma_\eta$, we have $|x|<\eta y$. Now, as $|z-t|^2 =
(z-t)(\bar z-t)$ and $z(\bar z-t) = |z|^2 - zt$, using \eqref{eqrG},
we have
\[
zr_G(z) = z \int_0^\infty\frac{t^{p+1}}{z-t} \,d\mu(t) = |z|^2 \int
_0^\infty\frac{t^{p+1}}{|z-t|^2} \,d\mu(t) - z \int_0^\infty\frac
{t^{p+2}}{|z-t|^2} \,d\mu(t),
\]
which gives
%
\begin{equation}
\Re(zr_G(z)) = |z|^2 \int_0^\infty\frac{t^{p+1}}{|z-t|^2} \,d\mu(t) -
\Re z \int_0^\infty\frac{t^{p+2}}{|z-t|^2} \,d\mu(t)\label{eqrealpart}
\end{equation}
and
\begin{equation}
\Im(zr_G(z)) = -\Im z \int_0^\infty\frac{t^{p+2}}{|z-t|^2} \,d\mu
(t).\label{eqimaginarypart}
\end{equation}

On $\Gamma_\eta$ and for $t\in[0,\eta y]$, $|t-x|\le t+|x|\le2\eta y$.
Thus, we have
%
\begin{eqnarray} \label{eqdivergence}
\int_0^\infty\frac{|z|^2 t^{p+1}}{|z-t|^2} \,d\mu(t)& \ge&\int
_0^{\eta y}
\frac{y^2 t^{p+1}}{(t-x)^2+y^2} \,d\mu(t)
\nonumber
\\[-8pt]
\\[-8pt]
\nonumber
&\ge&\frac1{1+4\eta^2}
\int
_0^{\eta y} t^{p+1} \,d\mu(t) \to\infty,
\end{eqnarray}
as $y\to\infty$, since $\mu\in\mathcal M_p$.

Now fix $\eta>0$, and consider a sequence $\{z_n=x_n+i y_n\}$ in
$\Gamma
_\eta$, such that $|z_n|\to\infty$, that is, $|x_n|\le\eta y_n$ and
$y_n\to\infty$. Assume toward contradiction that $\{|z_n r_G(z_n)|\}$
is a bounded sequence. Then both the real and the imaginary parts of
the sequence will be bounded. However, then the boundedness of the real
part and \eqref{eqrealpart} and~\eqref{eqdivergence} give
\[
\biggl|\Re z_n \int_0^\infty\frac{t^{p+2}}{|z_n-t|^2} \,d\mu(t)\biggr|\to
\infty.
\]
Then, using \eqref{eqimaginarypart} and the fact that $|\Re z| \le
\eta\Im z$ on $\Gamma_\eta$, we have
\[
\Im(z_nr_G(z_n)) \ge\frac1\eta\biggl|\Re z_n \int_0^\infty\frac
{t^{p+2}}{|z_n-t|^2} \,d\mu(t)\biggr|\to\infty,
\]
which contradicts the fact that the imaginary part of $\{z_nr_G(z_n)\}$
is bounded and completes the proof.
\end{pf}

We now consider the upper bound for $r_G(z)$. The result and the proof
of the following proposition are inspired by Lemma 5.2(iii) of \cite
{bercovici2000functions}.
\begin{proposition}\label{propupperboundrG}
Let $\mu$ be a probability measure in the class $\mathcal M_{p,\alpha}$
for some nonnegative integer $p$ and $\alpha\in(p,p+1]$. Then, for any
$\beta\in[0,(\alpha-p)/(\alpha-p+1))$, we have
%
%
\begin{equation}\label{equpperboundrG}
r_G(z)=\lito(z^{-\beta}) \qquad\mbox{as $z\to\infty$ n.t.}
\end{equation}
\end{proposition}

\begin{remark}
We consider the principal branch of logarithm of a complex number with
positive imaginary part, while defining the fractional powers in \eqref
{equpperboundrG} above and elsewhere.
\end{remark}

\begin{remark}
Note that \eqref{equpperboundrG} holds also for $p=\alpha$ with
$\beta=0$, which can be readily seen from Theorem 1.5 of \cite
{benaych2006taylor}.
\end{remark}

\begin{pf*}{Proof of Proposition \protect\ref{propupperboundrG}}
Define a measure $\rho_0$ as $d\rho_0(t)=t^pd\mu(t)$. Since $\mu\in
\mathcal M_p$, $\rho_0$ is a finite measure. Further, since $p<\alpha$,
using Theorem~1.6.5 of~\cite{binghamgoldieteugels1987}, we have
$\rho_0(y,\infty) \sim\frac\alpha{\alpha-p} y^p \mu(y, \infty)$,
which is regularly varying of index $-(\alpha-p)$.

Now fix $\eta>0$. It is easy to check that for $t\ge0$ and $z\in
\Gamma
_\eta$, $t/|z-t| < \sqrt{1+\eta^2}$. For $z=x+iy$, we have $|z-t|>y$
and hence for $t\in[0,y^{1/(\alpha-p+1)}]$, we have
$t/|z-t|<y^{-(\alpha
-p)/(\alpha-p+1)}$. Then, using \eqref{eqrG} and the definition
of~$\rho_0$,\looseness=-1
\begin{eqnarray*}
|r_G(z)| &\le&\int_0^{y^{{1}/{(\alpha-p+1)}}} \biggl|\frac{t}{z-t}
\biggr|\,d\rho_0(t) + \sqrt{1+\eta^2} \rho_0\bigl(y^{{1}/{(\alpha-p+1)}}, \infty\bigr)\\
&\le& y^{-(\alpha-p)/(\alpha-p+1)} \int_0^\infty t^p \,d\mu(t) + \sqrt
{1+\eta^2} \rho_0\bigl(y^{1/(\alpha-p+1)},\infty\bigr) = \lito (y^{-\beta})
\end{eqnarray*}\looseness=0
for any $\beta\in[0,(\alpha-p)/(\alpha-p+1))$, as the second term is
regularly varying of index $-(\alpha-p)/(\alpha-p+1)$. Further, for
$z=x+iy\in\Gamma_\eta$, we have $|z|=\sqrt{x^2+y^2}\le y\sqrt
{1+\eta
^2}$, and hence we have the required result.
\end{pf*}

Next we specialize to the asymptotic behavior of $r_G(iy)$, as $y\to
\infty$. Observe that
%
%
\begin{eqnarray}\label{rGy}
\Re r_G(iy)&=&-\int_0^\infty\frac{t^{p+2}}{t^2+y^2} \,d\mu(t)\quad \mbox
{and}
\nonumber
\\[-8pt]
\\[-8pt]
\nonumber
 \Im r_G(iy)&=&-y \int_0^\infty\frac{t^{p+1}}{t^2+y^2} \,d\mu(t).
\end{eqnarray}

\begin{proposition} \label{propcauchytail}
Let $\mu$ be a probability measure in the class $\mathcal M_p$.

If $\alpha\in(p,p+1)$, then the following statements are equivalent:
\begin{enumerate}[(iii)]
\item[(i)]\hypertarget{propcauchymeastail}$\mu$ has regularly varying tail of index $-\alpha$.
\item[(ii)]\hypertarget{propcauchyrealtail}$\Re r_G(iy)$ is regularly varying of index $-(\alpha-p)$.
\item[(iii)]\hypertarget{propcauchyimagtail}$\Im r_G(iy)$ is regularly varying of index $-(\alpha-p)$.
\end{enumerate}
If any of the above statements holds, then
\begin{eqnarray*}
\frac{\sin({\pi(\alpha-p)}/2)}{{\pi(p+2-\alpha)}/2} \Re
r_G(iy)&\sim&\frac{\cos({\pi(\alpha-p)}/2)}{{\pi
(p+1-\alpha)}/2}
\Im r_G(iy)\\
&\sim&-y^p \mu(y,\infty)\qquad \mbox{as $y\to\infty$.}
\end{eqnarray*}
Further, $\Re r_G(iy)\gg y^{-1}$ and $\Im r_G(iy)\gg y^{-1}$ as $y\to
\infty$.\vadjust{\goodbreak}

If $\alpha=p$, then statements (\hyperlink{propcauchymeastail}{\textup{i}})
and (\hyperlink{propcauchyimagtail}{\textup{iii}}) above are equivalent. Also, if either of the
statements holds, then
%
%
\begin{equation}\label{eqimagequiv}
\Im r_G(iy)\sim-\frac\pi2y^p \mu(y,\infty)\qquad \mbox{as $y\to\infty$.}
\end{equation}
Further, $\Im r_G(iy)\gg y^{-1}$ as $y\to\infty$.

If $\alpha=p+1$, then statements \textup{(\hyperlink{propcauchymeastail}{i})}
and \textup{(\hyperlink{propcauchyrealtail}{ii})} above are equivalent. Also, if
either of the statements holds, then
%
%
\begin{equation}\label{eqrealequiv}
\Re r_G(iy)\sim-\frac\pi2 y^p \mu(y,\infty)\qquad \mbox{as $y\to
\infty$.}
\end{equation}
Further, for any $\varepsilon>0$, $\Re r_G(iy) \gg y^{-(1+\varepsilon
)}$ as $y\to\infty$.
\end{proposition}

\begin{remark}
Note that for $\alpha=p+1$, $\Re r_G(iy)$ is regularly varying of index
$-1$, and the asymptotic lower bound $\Re r_G(iy)\gg y^{-1}$ need not
hold. This causes some difficulty in the proofs of Propositions \ref
{fraction-taylor} and~\ref{inverse-taylor}. The lack of the asymptotic
lower bound has to be compensated for by the stronger upper bound
obtained in Proposition~\ref{propupperboundrG}, which holds for
$\alpha=p+1$. This is reflected in condition (\hyperlink{R4}{R4$''$}) for the class
$\mathcal R_{p,\beta}$ with $\beta>0$, defined in Section~\ref{secC-Vreln}.
Further note that the situation reverses for $\alpha=p$, as
Proposition~\ref{propupperboundrG} need not hold. The case, where
$\alpha\in(p,p+1)$ is not an integer, is simple, as the asymptotic
lower bounds hold for both the real and imaginary parts of $r_G(iy)$
(Proposition~\ref{propcauchytail}), as well as the stronger
asymptotic upper bound works (Proposition~\ref{propupperboundrG}).
However, the case of noninteger $\alpha\in(p,p+1)$ is treated
simultaneously with the case $\alpha=p$ as the class $\mathcal
R_{p,0}$; cf. Section~\ref{secC-Vreln} in Propositions \ref
{fraction-taylor} and~\ref{inverse-taylor}.
\end{remark}

\begin{pf*}{Proof of Proposition \protect\ref{propcauchytail}}
The asymptotic lower bounds for the real and the imaginary parts of
$r_G(iy)$ are immediate from (\hyperlink{propcauchyrealtail}{ii}) and
(\hyperlink{propcauchyimagtail}{iii}), respectively. So, we only need to show \eqref
{eqrealequiv} and the equivalence between (\hyperlink{propcauchymeastail}{i}) and
(\hyperlink{propcauchyrealtail}{ii}) when $\alpha\in(p,p+1]$
and \eqref{eqimagequiv} and the equivalence between (\hyperlink{propcauchymeastail}{i}) and
(\hyperlink{propcauchyimagtail}{iii}) when $\alpha\in[p,p+1)$.

Let $d\rho_j(t)=t^{p+j}\,d\mu(t)$, for $j=1, 2$. Then, by Theorem 1.6.4
of~\cite{binghamgoldieteugels1987}, we have for $\alpha\in[p,p+1)$,
$\rho_1[0,y]\sim\alpha/(p+1-\alpha) y^{p+1} \mu(y,\infty)$, which is
regularly varying of index $p+1-\alpha\in(0,1]$, and for $\alpha\in
(p,p+1]$, $\rho_2[0,y]\sim\alpha/(p+2-\alpha) y^{p+2} \mu(y,\infty)$,
which is regularly varying of index $p+2-\alpha\in[1,2)$. Further,
from \eqref{rGy}, we get
\[
\Re r_G(iy) = - \int_{0}^{\infty} \frac1{t^2+y^2}\,d\rho_2(t)\quad \mbox
{and}\quad \Im r_G(iy) = -y \int_0^\infty\frac1{t^2+y^2} \,d\rho_1(t).
\]
Then the results follow immediately from Proposition~\ref{stieltjes}.
\end{pf*}

While asymptotic equivalences between $\Re r_G(iy)$ and tail of $\mu$
for $\alpha=p$ and $\Im r_G(iy)$ and tail of $\mu$ for $\alpha=p+1$ are
not true in general, we obtain the relevant asymptotic bounds in these
cases. We also obtain the exact asymptotic orders when $p=0$.\vadjust{\goodbreak}
\begin{proposition} \label{proprGiy}
Consider a probability measure $\mu$ in the class $\mathcal M_p$.

If $\mu$ has regularly varying tail of index $-p$, then for any
$\varepsilon>0$, $\Re r_G(iy)\gg y^{-\varepsilon}$ as $y\to\infty$.
Further, if $p=0$, then $\Re r_G(iy)\sim- \mu(y, \infty)$ as $y\to
\infty$.

If $\mu$ has regularly varying tail of index $-(p+1)$, then $\Im
r_G(iy)$ is regularly varying of index $-1$ and $y^{-1} \ll\Im
r_G(iy)\ll y^{-(1-\varepsilon)}$ as $y\to\infty$, for any
$\varepsilon>0$.
\end{proposition}

\begin{remark}
Note that $\Im r_G(iy)$ is regularly varying for the case $\alpha=p+1$
in contrast to $\Re r_G(iy)$ for the case $\alpha=p>0$. Further, for
the case $\alpha=p+1$, the lower bound for $\Im r_G(iy)$ is sharper
than that for $\Re r_G(iy)$ and coincides with that of $\Im r_G(iy)$
for the case $\alpha\in[p,p+1)$ discussed in Proposition~\ref{propcauchytail}.
\end{remark}

\begin{pf*}{Proof of Proposition \protect\ref{proprGiy}}
First consider the case where $\mu$ has regularly varying tail of index
$-p$. We use the notation $d\rho_0(t)=t^pd\mu(t)$ introduced in the
proof of Proposition~\ref{propupperboundrG}. However, in the
current situation Theorem 1.6.4 of~\cite{binghamgoldieteugels1987}
will not apply. If $p=0$, then $\rho_0=\mu$ and $\rho_0(y,\infty)$ is
slowly varying. If $p>0$, observe that, as $\int t^p \,d\mu(t) <\infty$,
we have
\[
\rho_0(y,\infty)=y^p\mu(y,\infty)+p\int_0^y s^{p-1}\mu(s,\infty
)\,ds\sim
p\int_0^y s^{p-1}\mu(s,\infty)\,ds,
\]
which is again slowly varying, where we use Theorem 0.6(a) of \cite
{resnickbook}. Thus, in either case, $\rho_0(y,\infty)$ is slowly
varying and converges to zero as $y\to\infty$. Now, from \eqref{rGy}
and Proposition~\ref{stieltjes2}, we also have
\[
\Re r_G(iy) = -\int_0^\infty\frac{t^2}{t^2+y^2} \,d\rho_0(t) \sim
-\rho
_0(y,\infty)
\]
as $y\to\infty$. Since $\rho_0(y,\infty)$ is slowly varying, for any
$\varepsilon>0$, we have $|y^\varepsilon\times \Re r_G(iy)|\to\infty$ as
$y\to
\infty$. Also, for $p=0$, we have
\[
\Re r_G(iy) \sim-\rho_0(y,\infty)=-\mu(y,\infty).
\]

Next consider the case where $\mu\in\mathcal M_p$ has regularly varying
tail of index $-(p+1)$. Define again $d\rho_1(t) = t^{p+1} \,d\mu(t)$.
Then,
\[
\rho_1[0,y] = (p+1) \int_0^y s^p \mu(s,\infty) \,ds - y^{p+1} \mu
(y,\infty
) \sim(p+1) \int_0^y s^p \mu(s,\infty) \,ds
\]
is slowly varying, again by Theorem 0.6(a) of~\cite{resnickbook}. Then,
by \eqref{rGy} and Proposition~\ref{stieltjes}, we have
\[
\Im r_G(iy) = y \int_0^\infty\frac{d\rho_1(t)}{t^2+y^2} \sim\frac1y
\rho_1[0,y]
\]
is regularly varying of index $-1$. Further, $\rho_1[0,y] \to\int
_0^\infty t^{p+1} \,d\mu(t) = \infty$ as $y\to\infty$. Then the
asymptotic upper and lower bounds follow immediately.~%
\end{pf*}

\section{Relationship between Cauchy and Voiculescu transforms} \label
{secC-Vreln}
The results of the previous section relate the tail of a regularly
varying probability measure and the behavior of the remainder term in
Laurent series\vadjust{\goodbreak} expansion of its Cauchy transform. In this section, we
shall relate the remainder terms in Laurent series expansion of Cauchy
and Voiculescu transforms. Finally, we collect the results from
Sections~\ref{secCauchy} and~\ref{secC-Vreln} to prove
Theorems~\ref{thmerrorequiv}--\ref{thmerrorequivnew}.

To study the relation between the remainder terms in Laurent series
expansion of Cauchy and Voiculescu transforms, we consider a class of
functions, which include the functions $H_\mu$ for the probability
measures $\mu$ with regularly varying tails. We then show that the
class is closed under appropriate operations. See Propositions~\ref
{fraction-taylor} and~\ref{inverse-taylor}.

Let $\mathcal H$ denote the set of analytic functions $A$ having a
domain $\mathcal{D}_A$ such that for all positive $\eta$, there exists
$\delta>0$ with $\Delta_{\eta,\delta}\subset\mathcal{D}_A$.

For a nonnegative integer $p$ and $\beta\in[0,1/2)$, let $\mathcal
{R}_{p,\beta}$ denote the set of all functions $A\in\mathcal H$ which
satisfy the following conditions:
\begin{enumerate}[(R1)]
\item[(R1)]\hypertarget{R1} $A$ has Taylor series expansion with real coefficients of the form
\[
A(z) = z + \sum_{j=1}^p a_j z^{j+1} + z^{p+1} r_A(z),
\]
where $a_1, \ldots, a_p$ are real numbers. For $p=0$, we interpret the
sum in the middle term as absent.
\item[(R2)]\hypertarget{R2}$z\ll r_A(z)\ll z^\beta$ as $z\to0$ n.t.
\item[(R3)]\hypertarget{R3}$\Re r_A(-iy)\gg y^{1+\beta/2}$ and $\Im r_A(-iy)\gg y$ as $y\to
0+$.
\end{enumerate}
For $p=0=\beta$, we further require that
\begin{enumerate}[(R4$'$)]
\item[(R4$'$)]\hypertarget{R4prime}$\Re r_A(-iy) \approx\Im r_A(-iy)$ as $y\to0+$.
\end{enumerate}
For $\beta\in(0,1/2)$, we further require that,
\begin{enumerate}[(R4$''$)]
\item[(R4$''$)]\hypertarget{R4}$\Re r_A(-iy)\ll y^{1-\beta/2} \mbox{ and } \Im
r_A(-iy)\ll y^{1-\beta/2} \mbox{ as $y\to0+$.}$
\end{enumerate}
Note that the functions in $\mathcal R_{p,\beta}$ satisfy (\hyperlink{R1}{R1})--(\hyperlink{R3}{R3}) for $p\geq1$. For $p=0=\beta$, the functions in
$\mathcal R_{p,\beta}$ satisfy (\hyperlink{R1}{R1})--(\hyperlink{R3}{R3}) as well
as (\hyperlink{R4prime}{R4$'$}). Finally, for nonnegative integers $p$ and $\beta\in
(0,1/2)$, the functions in $\mathcal R_{p,\beta}$ satisfy (\hyperlink{R1}{R1})--(\hyperlink{R3}{R3}) and (\hyperlink{R4}{R4$''$}).

The classes $\mathcal R_{p,\beta}$ as $p$ varies over the set of
nonnegative integers and $\beta$ varies over $[0,1/2]$, include the
functions $H_\mu$ where $\mu\in\mathcal M_{p,\alpha}$ with $p$ varying
over nonnegative integers and $\alpha$ varying over $[p,p+1]$:

\textit{Case} I: $p$ positive integer and $\alpha\in[p,p+1)$: By
Proposition~\ref{proplowerbdrG} and~\ref{propcauchytail}, we
have $H_\mu\in\mathcal R_{p,0}$.

\textit{Case} II: $p=0$, $\alpha\in[0,1)$: By Proposition \ref
{proplowerbdrG},~\ref{propcauchytail} and~\ref{proprGiy}, $H_\mu\in
\mathcal R_{0,0}$.

Proposition~\ref{proprGiy} is required to prove (\hyperlink{R4prime}{R4$'$}) for
$p=\alpha=0$ only.

\textit{Case} III: $p$ nonnegative integer, $\alpha=p+1$: By
Proposition~\ref{propupperboundrG} and~\ref{proprGiy}, $H_\mu$
will be in $\mathcal R_{p,\beta}$ for any $\beta\in(0,1/2)$.

We do not impose the condition $\Re r_A(-iy) \approx\Im r_A(-iy)$ for
$p>0$, as it may fail for some measures in $\mathcal M_{p,p}$.

The first result deals with the reciprocals. Note that $U(z)$ and
$zU(z)$ have the same remainder functions, and if one belongs to the
class $\mathcal H$, so does the other.

\begin{proposition}\label{fraction-taylor}
Suppose $zU(z)\in\mathcal{H}$ be a function belonging to $\mathcal
R_{p, \beta}$ for some nonnegative integer $p$ and $0\le\beta<1/2$,
such that $U$ does not vanish in a neighborhood of zero. Then the
reciprocal $V=1/U$ is defined and $zV(z)$ is also in $\mathcal
R_{p,\beta}$. Furthermore, we have:
\begin{enumerate}[(F1)]
\item[(F1)]\hypertarget{taylor1}$r_V(z)\sim-r_U(z)$, as $z\to0$ n.t.;
\item[(F2)]\hypertarget{taylor2}$\Re r_V(-iy)\sim-\Re r_U(-iy)$, as $y\rightarrow0+$;
\item[(F3)]\hypertarget{taylor3}$\Im r_V(-iy)\sim-\Im r_U(-iy)$, as $y\rightarrow0+$.
\end{enumerate}
\end{proposition}

The second result shows that for each of the above classes, when we
consider a bijective function from the class, its inverse is also in
the same class.\looseness=-1

\begin{proposition}\label{inverse-taylor}
Suppose $U\in\mathcal{H}$ be a bijective function with the inverse in~$\mathcal{H}$
as well and $U\in\mathcal R_{p,\beta}$ for some
nonnegative integer $p$ and $0\le\beta<1/2$. Then the inverse $V$ is
defined and is also in $\mathcal R_{p,\beta}$. Furthermore, we have:
\begin{enumerate}[(I3)]
\item[(I1)]\hypertarget{inverse-taylor1}$r_V(z)\sim-r_U(z)$, as $z\to0$ n.t.;
\item[(I2)]\hypertarget{inverse-taylor2}$\Re r_V(-iy)\sim-\Re r_U(-iy)$, as
$y\rightarrow0+$;
\item[(I3)]\hypertarget{inverse-taylor3}$\Im r_V(-iy)\sim-\Im r_U(-iy)$, as $y\rightarrow0+$.
\end{enumerate}
\end{proposition}

Next we prove Propositions~\ref{fraction-taylor} and \ref
{inverse-taylor}. In both the proofs, all the limits will be taken as
$z\to0$ n.t. or $y\to0+$, unless otherwise mentioned, and these
conventions will not be stated repeatedly. We shall also use that for
any nonnegative integer $p$ and $\beta\in[0,1/2)$, with $U\in
\mathcal
R_{p,\beta}$, we have
%
%
\begin{equation} \label{eqrealimag}\qquad
|\Re r_U(-iy)|\le|r_U(-iy)|\ll1\quad \mbox{and}\quad |\Im r_U(-iy)|\le
|r_U(-iy)|\ll1.
\end{equation}
The proofs of Propositions~\ref{fraction-taylor} and \ref
{inverse-taylor} will be broken down into cases $p=0$ and $p\geq1$.
Each of these cases will be further split into subcases $\beta=0$ and
$\beta\in(0,1/2)$. The $p\geq1$ is more involved compared to the case
$p=0$. However, the proofs, specially that of Proposition \ref
{inverse-taylor}, have substantial parts in common for different cases.

We first prove the result regarding the reciprocal.
\begin{pf*}{Proof of Proposition~\ref{fraction-taylor}}
Note that since $zU(z)$ belongs to $\mathcal H$, given $\eta>0$, there
exists $\delta>0$, such that $\Delta_{\eta,\delta}$ is contained in the
domain of $zU(z)$, and $U$ does not vanish on $\Delta_{\eta,\delta}$.
Thus, $V(z)$ and hence $zV(z)$ will also be defined on $\Delta_{\eta
,\delta}$. So $zV(z)$ also belongs to $\mathcal H$.

Observe that if we verify (\hyperlink{taylor1}{F1})--(\hyperlink{taylor3}{F3}), then
$zV(z)$ is automatically in $\mathcal R_{p,\beta}$ as well, since
$V(z)$ and $zV(z)$ have same remainder functions. We shall prove (\hyperlink{taylor1}{F1})--(\hyperlink{taylor3}{F3})
using the fact that $V(z)=1/U(z)$ and the
properties of $zU(z)$ as an element of $\mathcal R_{p,\beta}$.\vadjust{\goodbreak}

\textit{Case} I: $p=0$. Let $zU(z)=z+zr_U(z)$ be a function in this class.
Then $V(z)=1-r_U(z)+\bigo(|r_U(z)|^2)$. By uniqueness of Taylor's
expansion from Lemma A.1 of~\cite{benaych2006taylor}, we have
%
%
\begin{equation} \label{eqrecipremainder}
r_V(z)=-r_U(z)+\bigo(|r_U(z)|^2).
\end{equation}
Since, by (\hyperlink{R2}{R2}), $r_U(z)\ll1$, we have $r_V(z)\sim-r_U(z)$, which
checks (\hyperlink{taylor1}{F1}).

Further, evaluating \eqref{eqrecipremainder} at $z=-iy$ and equating
the real and the imaginary parts, we have
\[
\Re r_V(-iy)=-\Re r_U(-iy)+\bigo(|r_U(-iy)|^2)
\]
and
\[
\Im
r_V(-iy)=-\Im r_U(-iy)+\bigo(|r_U(-iy)|^2).
\]
Thus, to obtain the equivalences (\hyperlink{taylor2}{F2}) and (\hyperlink{taylor3}{F3}),
it is enough to show that $|r_U(-iy)|^2 = |\Re r_U(-iy)|^2 + |\Im
r_U(-iy)|^2$ is negligible with respect to both the real and the
imaginary parts of $r_U(-iy)$. We prove the negligibility separately
for two subcases $\beta=0$ and $\beta\in(0,1/2)$.

\textit{Subcase} Ia: $p=0$, $\beta=0$. Using \eqref{eqrealimag} and
$\Re r_U(-iy) \approx\Im r_U(-iy)$ from (\hyperlink{R4prime}{R4$'$}), we have the
required negligibility condition.

\textit{Subcase} Ib: $p=0$, $\beta\in(0,1/2)$. Using (\hyperlink{R3}{R3})
and (\hyperlink{R4}{R4$''$}), we have
\[
\frac{|\Im r_U(-iy)|^2}{|\Re r_U(-iy)|} = \frac{y^{1+\beta/2}}{|\Re
r_U(-iy)|} \biggl(\frac{|\Im r_U(-iy)|}{y^{1-\beta/2}}\biggr)^2
y^{1-3\beta/2}\to0
\]
and
\[
\frac{|\Re r_U(-iy)|^2}{|\Im r_U(-iy)|} = \frac{y}{|\Im r_U(-iy)|}
\biggl(\frac{|\Re r_U(-iy)|}{y^{1-\beta/2}}\biggr)^2 y^{1-\beta}\to0.
\]
They, together with \eqref{eqrealimag}, give the required negligibility
condition, thus proving~(\hyperlink{taylor2}{F2}) and
(\hyperlink{taylor3}{F3}).

\textit{Case} II: $p\geq1$. Let $zU(z)= z + \sum_{j=1}^p u_j z^{j+1} +
z^{p+1} r_U(z)$ be a function in this class. Note that, as $p\ge1$ and
by (\hyperlink{R2}{R2}), as $z\ll r_U(z)$, we have $\sum_{j=1}^p u_j z^j + z^p
r_U(z) = u_1 z + \bigo( z r_U(z) )$. Thus, using (\hyperlink{R2}{R2}), we have
\begin{eqnarray*}
V(z)&=&1+\sum_{j=1}^p (-1)^j \Biggl(\sum_{m=1}^p u_m z^{m} + z^{p}
r_U(z)\Biggr)^j\\
&&{}+ (-1)^{p+1}u_1^{p+1} z^{p+1} + \bigo( z^{p+1} r_U(z) ).
\end{eqnarray*}

Now we expand the second term on the right-hand side. As $z\ll r_U(z)$
from~(\hyperlink{R2}{R2}), all powers of $z$ with indices greater than $(p+1)$
can be absorbed in the last term on the right-hand side. Then collect
the $(p+1)$th powers of $z$ in the second and third terms to get $c_1
z^{p+1}$ for some real number $c_1$. The remaining powers of $z$ form a
polynomial $P(z)$ of degree at most $p$ with real coefficients. Finally
we consider the terms containing some power of $r_U(z)$. It will
contain terms of the form $z^{l_1} (z^p r_U(z))^{l_2}$ for integers
$l_1\ge0$ and $l_2\ge1$, with the leading term being $-z^p r_U(z)$.
Since $p\ge1$ and from (\hyperlink{R2}{R2}) we have $r_U(z)\ll1$, the remaining
terms can be absorbed in the last term on the right-hand side. Thus we get
\[
V(z)=1+P(z)-z^pr_U(z)+ c_1 z^{p+1} + \bigo(z^{p+1} r_U(z)).
\]
By uniqueness of Taylor series expansion from Lemma A.1 of \cite
{benaych2006taylor}, we have
\[
r_V(z) = -r_U(z)+c_1 z + \bigo(zr_U(z)).
\]

The form of $r_V$ immediately gives $r_V(z) \sim-r_U(z)$, since $z\ll
r_U(z)$, by (\hyperlink{R2}{R2}). This proves (\hyperlink{taylor1}{F1}).

Also, using \eqref{eqrealimag}, $\Im r_V(-iy) = -\Im r_U(-iy) +
\bigo(y)$ and as $y\ll\Im r_U(-iy)$ from (\hyperlink{R3}{R3}), we have $\Im
r_V(-iy) \sim-\Im r_U(-iy)$. This shows (\hyperlink{taylor3}{F3}). Further, as
$c_1$ is real, $\Re r_V(-iy) =-\Re r_U(-iy) + \bigo( y |r_U(-iy)|)$.
Thus, to conclude (\hyperlink{taylor2}{F2}), it is enough to show that $y
|r_U(-iy)| \ll\Re r_U(-iy)$, for which it is enough to show that $y
\Im r_U(-iy) \ll\Re r_U(-iy)$. We show this separately for two subcases.

\textit{Subcase} IIa: $p\geq1$, $\beta=0$. We have by (\hyperlink{R3}{R3}),
\[
\frac{y \Im r_U(-iy)}{\Re r_U(-iy)} =
\frac{y}{\Re r_U(-iy)} \cdot\Im r_U(-iy).
\]

\textit{Subcase} IIb: $p\geq1$, $\beta\in(0,1/2)$. Using the
properties (\hyperlink{R3}{R3}) and (\hyperlink{R4}{R4$''$}) we get
\[
\frac{y \Im r_U(-iy)}{\Re r_U(-iy)} =
\frac{y^{1+\beta/2}}{\Re r_U(-iy)} \cdot\frac{\Im
r_U(-iy)}{y^{1-\beta
/2}} \cdot y^{1-\beta}.
\]

It is easy to see that the limit is zero in either subcase.
\end{pf*}

Before proving the result regarding the inverse, we provide a result
connecting a function in the class $\mathcal H$ and its derivative.
\begin{lemma}\label{lemderiv}
Let $v\in\mathcal H$ satisfy $v(z)=\lito(z^\beta)$ as $z\to0$ n.t., for
some real number $\beta$. Then $v^\prime(z) = \lito(z^{\beta-1})$ as
$z\to0$ n.t.
\end{lemma}
\begin{pf}
The result for $\beta=0$ follows from the calculations in the proof of
Proposition A.1(ii) of~\cite{benaych2006taylor}. For the general case,
define $w(z) = z^{-\beta} v(z)$. Then $w\in\mathcal H$ and
$w(z)=\lito
(1)$. So by the case $\beta=0$, we have $w^\prime(z) = -\beta
z^{-\beta
-1} v(z) + z^{-\beta} v^\prime(z)=\lito(z^{-1})$. Thus, $z w^\prime(z)
= -\beta z^{-\beta} v(z) + z^{-(\beta-1)} v^\prime(z)$, where the
left-hand side and the first term on the right-hand side are $\lito
(1)$, and hence the second term on the right-hand side is $\lito(1)$
as well.
\end{pf}

We are now ready to prove the result regarding the inverse.\vadjust{\goodbreak}
\begin{pf*}{Proof of Proposition \protect\ref{inverse-taylor}} We begin with
some estimates which work for all values of $p$ and $\beta$ before
breaking into cases and subcases.
Since $U$ is of the form
\[
U(z) = z + \sum_{j=1}^p u_j z^{j+1} + z^{p+1} r_U(z)
\]
and $r_U(z)\ll1$, by Proposition A.3 of~\cite{benaych2006taylor}, the
inverse function $V$ also has the same form with the remainder term
$r_V$ satisfying
%
%
\begin{equation}\label{eqrv}
r_V(z)\ll1.
\end{equation}
Also note that $V(z)\sim z$. Further, Lemma A.1 of \cite
{benaych2006taylor} shows that the coefficients are determined by the
limits of the derivatives of the function at $0$. Hence, the real
coefficients of $U$ guarantee that the coefficients of $V$ are real. So
we only need to check the asymptotic equivalences of the remainder
functions given in (\hyperlink{inverse-taylor1}{I1})--(\hyperlink{inverse-taylor3}{I3}). We
shall achieve this by analyzing $I(z)=r_U(V(z))-r_U(z)$, the fact that
$U(V(z))=z$ and the properties of $U$ as an element in $\mathcal
R_{p,\beta}$. For that purpose, we define
\[
I(z)=r_U(V(z))-r_U(z)=\int_{\gamma_z} r_U^\prime(\zeta) \,d\zeta,
\]
where $\gamma_z$ is the closed line segment joining $z$ and $V(z)$.
Using the part (a) in the proof of Proposition A.3 of \cite
{benaych2006taylor}, given any $\eta>0$, we have for all small enough
$\delta>0$,
\[
\Delta_{2\eta, 2\delta} \subset\mathcal D_U\quad \mbox{and}\quad
V(\Delta_{\eta, \delta}) \subset\Delta_{2\eta, 2\delta}.
\]
Thus, given any $\eta>0$, there exists $\delta>0$, such that whenever
$z\in\Delta_{\eta, \delta}$, $V(z)$ belongs to $\Delta_{2\eta,
2\delta
}$. Note that $\Delta_{2\eta, 2\delta}$ is a convex set. Hence,
whenever $z\in\Delta_{\eta,\delta}$, $\gamma_z$ is contained in
$\Delta
_{2\eta, 2\delta}\subset\mathcal D_U$, and $r_U^\prime$ is defined on
the entire line segment $\gamma_z$. We shall need the following
estimate, that
\[
|I(z)| \le|\gamma_z| \sup_{\zeta\in\gamma_z} |r_U^\prime(\zeta
)| =
|V(z)-z| \sup_{\zeta\in\gamma_z} |r_U^\prime(\zeta)| = |V(z)-z|
|r_U^\prime(\zeta_0(z))|,
\]
for some $\zeta_0(z)\in\gamma_z$, since $\gamma_z$ is compact. Note
that $\zeta_0(z)=z+\theta(z)(V(z)-z)$, for some $\theta(z)\in[0,1]$ and
hence $\zeta_0(z)\sim z$. Now, $r_U(z)=\lito(z^\beta)$ by (\hyperlink{R2}{R2}),
and thus, by Lemma~\ref{lemderiv}, we have $r_U^\prime(\zeta
_0(z)) =
\lito(\zeta_0(z)^{\beta-1}) = \lito(z^{\beta-1})$. Further estimates
for $I(z)$ depend on the functions of $V(z)$ which are separate for the
cases $p=0$ and $p\geq1$. Using $V(z)=z+zr_V(z)$ for $p=0$ and
$V(z)=z+\bigo(z^2)$ for $p\ge1$, we have
%
%
\begin{equation}\label{eqIest}
|I(z)| =
\cases{
\lito(z^\beta r_V(z)), &\quad $\mbox{for $p=0$,}$\vspace*{2pt}\cr
\lito(z^{1+\beta}), &\quad $\mbox{for $p\ge1$.}$}
\end{equation}

\textit{Case} I: $p=0$. Then $U(z)=z+zr_U(z)$ and $V(z)=z+zr_V(z)$. Using
$U(V(z))=z$ and $I(z)=r_U(V(z))-r_U(z)$,\vadjust{\goodbreak} we get $0 = zr_V(z) +
(z+zr_V(z))\times (r_U(z)+I(z))$. Further canceling $z$ and using \eqref{eqrv}, we have
%
%
\begin{equation}\label{eqremp0}
0=r_U(z)+r_V(z)+r_U(z)r_V(z)+\bigo(I(z)).
\end{equation}
Using \eqref{eqIest} for $p=0$ and $r_U(z)\ll1$ from (\hyperlink{R2}{R2}), we
have $r_V(z)\sim-r_U(z)$, which proves (\hyperlink{inverse-taylor1}{I1}).
Further, using (\hyperlink{R2}{R2}) and evaluating at $z=-iy$, we have, for
$\beta
\in[0,1/2)$,
%
%
\begin{equation} \label{eqr0beta}
|r_V(-iy)|\ll y^\beta.
\end{equation}
Evaluating \eqref{eqremp0} at $z=iy$ and equating the real and the
imaginary parts, we have
%
\begin{equation}
\qquad 0=\Re r_U(-iy)+\Re r_V(-iy) + \bigo(|r_U(-iy)| |r_V(-iy)|) + \bigo
(|I(-iy)|)\label{eqinvrealp0}
\end{equation}
and
\begin{equation}
\qquad 0=\Im r_U(-iy)+\Im r_V(-iy) + \bigo(|r_U(-iy)| |r_V(-iy)|) + \bigo
(|I(-iy)|).\label{eqinvimagp0}
\end{equation}
We split the proofs of (\hyperlink{inverse-taylor2}{I2}) and (\hyperlink{inverse-taylor3}{I3}) for the case $p=0$ into further subcases $\beta=0$
and $\beta\in(0,1/2)$.

\textit{Subcase} Ia: $p=0$, $\beta=0$. By (\hyperlink{inverse-taylor1}{I1}) for
$z=-iy$ and (\hyperlink{R4prime}{R4$'$}), we have
\[
|I(-iy)|\ll|r_V(-iy)|\sim|r_U(-iy)|\approx|\Re r_U(-iy)|\approx|\Im
r_U(-iy)|.
\]
Thus, the last term on the right-hand side of \eqref{eqinvrealp0}
and \eqref{eqinvimagp0} are negligible with respect to $\Re
r_U(-iy)$ and $\Im r_U(-iy)$, respectively. Then, further using
$r_U(-iy)\to0$ from (\hyperlink{R2}{R2}), the third term on the right-hand side
of \eqref{eqinvrealp0} and \eqref{eqinvimagp0} are negligible
with respect to $\Re r_U(-iy)$ and $\Im r_U(-iy)$, respectively, and
hence we get $\Re r_U(-iy)\sim-\Re r_V(-iy)$ and $\Im r_U(-iy)\sim
-\Im
r_V(-iy)$, which prove (\hyperlink{inverse-taylor2}{I2}) and~(\hyperlink{inverse-taylor3}{I3}).

\textit{Subcase} Ib: $p=0$, $\beta\in(0,1/2)$. We have, by (\hyperlink{R3}{R3})
and~(\hyperlink{R4}{R4$''$}),
\[
y^\beta\frac{|\Im r_U(-iy)|}{|\Re r_U(-iy)|} = \frac{|\Im
r_U(-iy)|}{y^{1-\beta/2}} \frac{y^{1+\beta/2}}{|\Re r_U(-iy)|} \to0
\]
and
\[
y^\beta\frac{|\Re r_U(-iy)|}{|\Im r_U(-iy)|} = \frac{|\Re
r_U(-iy)|}{y^{1-\beta/2}} \frac{y}{|\Im r_U(-iy)|} y^{\beta/2} \to0.
\]
They, together with \eqref{eqrealimag}, give $y^\beta|r_U(-iy)|$ which is
negligible with respect to both the real and the imaginary parts of
$r_U(-iy)$. Further, using \eqref{eqIest} and \eqref{eqr0beta},
respectively, we have
\begin{eqnarray*}
|I(-iy)|&\ll& y^\beta|r_V(-iy)| \sim y^\beta|r_U(-iy)| \quad\mbox{and}\\
 |r_U(-iy) r_V(-iy)| &\ll& y^\beta|r_U(-iy)|.
\end{eqnarray*}
Thus, both $|I(-iy)|$ and $|r_U(-iy) r_V(-iy)|$ which are the last two
terms of~\eqref{eqinvrealp0} and \eqref{eqinvimagp0}, are
negligible with respect to both the real and the imaginary parts of
$r_U(-iy)$. Then, from \eqref{eqinvrealp0} and \eqref{eqinvimagp0}, we immediately have $\Re r_U(-iy)\sim-\Re r_V(-iy)$ and $\Im
r_U(-iy)\sim-\Im r_V(-iy)$, which prove (\hyperlink{inverse-taylor2}{I2})
and~(\hyperlink{inverse-taylor3}{I3}).

\textit{Case} II: $p\ge1$. In this case $U(z)=z+\sum
_{j=1}^pu_jz^{j+1}+z^{p+1}r_U(z)$ and $V(z) = z + \sum
_{j=1}^pv_jz^{j+1} + z^{p+1}r_V(z) = z(1+v_1z(1+\lito(1)))$. Using
$z=U(V(z))$ and canceling $z$ on both sides, we have
%
\begin{eqnarray} \label{eqcompose}
\qquad 0 &=& \sum_{j=1}^pv_jz^{j+1} + z^{p+1}r_V(z) + \sum_{m=1}^p u_m \Biggl( z
+ \sum_{j=1}^pv_jz^{j+1} + z^{p+1}r_V(z) \Biggr)^{m+1}
\nonumber
\\[-8pt]
\\[-8pt]
\nonumber
&&{} + z^{p+1}
\bigl(r_U(z)+I(z)\bigr) \bigl(1+(p+1)v_1z\bigl(1+\lito(1)\bigr)\bigr).
\end{eqnarray}
Note that all the coefficients on the right-hand side are real. We
collect the powers of $z$ up to degree $p+1$ on the right side in the
polynomial $Q(z)$. Let $c'\in\mathbb R$ be the coefficient of $z^{p+2}$
on the right side. The remaining powers of $z$ on the right-hand side
will be $\bigo(z^{p+3})$. We next consider the terms with $r_V(z)$ as a
factor and observe that $z^{p+1} r_V(z)$ is the leading term and the
remaining terms contribute $\bigo(z^{p+2} r_V(z))$. Finally, the last
term on the right-hand side gives $z^{p+1} r_U(z) + \bigo(z^{p+2}
r_U(z)) + \bigo(z^{p+1} I(z))$. Since $z\ll r_U(z)$ by (\hyperlink{R2}{R2}), the
term $\bigo(z^{p+3})$ can be absorbed in $\bigo(z^{p+2} r_U(z))$.
Combining the above facts and dividing \eqref{eqcompose} by
$z^{p+1}$, we get
%
%
\begin{eqnarray}
\label{eqcasep1aftercompose}
0&=&z^{-(p+1)} Q(z) + \bigl( r_U(z) + c' z + \bigo(I(z)) + \bigo(z r_U(z))
\bigr)
\nonumber
\\[-8pt]
\\[-8pt]
\nonumber
&&{}+ \bigl( r_V(z) + \bigo(z r_V(z)) \bigr).
\end{eqnarray}

As $I(z)\ll z^{1+\beta}\ll z\ll r_U(z)$ by \eqref{eqIest}
and (\hyperlink{R2}{R2}), we have $r_U(z)+c'z+\bigo(I(z))+\bigo(zr_U(z))=r_U(z)(1+\lito
(1))$. Also $r_V(z)+\bigo(zr_V(z))=r_V(z)(1+\lito(1))$. Thus, the last
two terms on the right-hand side of \eqref{eqcasep1aftercompose} goes
to zero. However, the first term on the right-hand side of \eqref{eqcasep1aftercompose}, $Q$ being a polynomial of degree at most $p$,
becomes unbounded unless $Q\equiv0$. So we must have $Q\equiv0$. Thus,
 \eqref{eqcasep1aftercompose} simplifies to
%
%
\begin{equation}\label{eqcompose2}
r_U(z) + c' z + \bigo(I(z)) + \bigo(z r_U(z)) = - r_V(z) + \bigo(z r_V(z)).
\end{equation}
As observed earlier, the left-hand side is $r_U(z) (1+\lito(1))$, and
the right-hand side is $- r_V(z) (1+\lito(1))$, giving $r_U(z) \sim-
r_V(z)$, which proves (\hyperlink{inverse-taylor1}{I1}).

Further, as in the case $p=0$, we have \eqref{eqr0beta} from
$r_U(z)\sim-r_V(z)$. Also, \eqref{eqcompose2} becomes
%
%
\begin{equation} \label{eqcompose3}
-r_V(z) = r_U(z) + c' z + \bigo(I(z)) + \bigo(z r_U(z)).
\end{equation}
Evaluating \eqref{eqcompose3} at $z=-iy$ and equating the imaginary
parts, we have, using~\eqref{eqIest},
\[
-\Im r_V(-iy) = \Im r_U(-iy) + \bigo(y).
\]
This gives (\hyperlink{inverse-taylor3}{I3}), that is, $-\Im r_V(-iy) \sim\Im
r_U(-iy)$, since $y\ll\Im r_U(-iy)$ by~(\hyperlink{R3}{R3}).\vadjust{\goodbreak}

Evaluating \eqref{eqcompose2} at $z=-iy$ again and now equating the
real parts, we have, as~$c'$ is real,
\[
-\Re r_V(-iy) = \Re r_U(-iy) + \bigo(|I(-iy)|) + \bigo(y|r_U(-iy)|).
\]
From \eqref{eqIest} and (\hyperlink{R3}{R3}), we have $|I(-iy)|\ll
y^{1+\beta}
\ll\Re r_U(-iy)$. Thus, to obtain~(\hyperlink{inverse-taylor2}{I2}), that is,
$-\Re r_V(-iy) \sim\Re r_U(-iy)$, we only need to show that\break
$y|r_U(-iy)|\ll\Re r_U(-iy)$, which follows using $r_U(-iy)\sim
-r_V(-iy)$, \eqref{eqr0beta} and (\hyperlink{R3}{R3}), since
\[
\frac{y|r_U(-iy)|}{|\Re r_U(-iy)|} = \frac{y^{1+\beta/2}}{|\Re
r_U(-iy)|} \frac{|r_U(-iy)|}{y^\beta} y^{\beta/2}.
\]
\upqed\end{pf*}

We wrap up the article by collecting the results from Sections \ref
{secCauchy} and~\ref{secC-Vreln} and proving Theorems~\ref{thmerrorequiv}--\ref{thmerrorequivnew}.
\begin{pf*}{Proofs of Theorems
\protect\ref{thmerrorequiv}--\protect\ref{thmerrorequivnew}}
We shall prove all the theorems together, as the proofs are very similar.

The statements involving the tail of the probability measure $\mu$ and
the remainder term in Laurent expansion of Cauchy transform, $r_{G_\mu
}$ can be obtained from the results in Section~\ref{secCauchy} as follows:
For all the theorems, the equivalence of the
statements (\hyperlink{tail}{i})
and (\hyperlink{Cauchyremainder}{ii}) about the tail of the probability measure
and Cauchy transform (the imaginary part in Theorems~\ref{thmerrorequiv}--\ref{thmerrorequiv-0}
and the real part in Theorem~\ref{thmerrorequivnew}) are given in Proposition~\ref{propcauchytail}. The
asymptotic equivalences between the tail of the measure and (the real
and the imaginary parts of) the remainder term in Laurent series
expansion of Cauchy transform, given in \eqref{imrg-rphi}, \eqref
{rerg-rphi}, \eqref{imrg-rphieqp}, \eqref{imrg-rphi-0} and \eqref
{rerg-rphinew} are also given in Proposition~\ref{propcauchytail}.
The similar asymptotic equivalence in \eqref{rerg-rphi-0} follows from
Propositions~\ref{propcauchytail} and~\ref{proprGiy} for the
cases $\alpha\in(0,1)$ and $\alpha=1$ respectively. We consider the
asymptotic upper and lower bounds next. The asymptotic lower bounds
in \eqref{rg-rphi}, \eqref{rg-rphieqp}, \eqref{rg-rphi-0}
and \eqref
{rg-rphinew} follow from Proposition~\ref{proplowerbdrG}. The
asymptotic upper bound in \eqref{rg-rphinew} follows from
Proposition~\ref{propupperboundrG}. The asymptotic lower bounds
in \eqref{imrg-rphi}, \eqref{rerg-rphi}, \eqref{imrg-rphieqp}, \eqref
{imrg-rphi-0} and \eqref{rerg-rphinew} follow from Proposition \ref
{propcauchytail}. The asymptotic upper bound in \eqref{rerg-rphinew} follows from the fact that $y^p \mu(y,\infty)$ is a regularly
varying function of index $-1$. The asymptotic lower bound in \eqref
{rerg-rphieqp} follows from Proposition~\ref{proprGiy}, while the
asymptotic lower bound in \eqref{rerg-rphi-0} follows as the tail of
the measure is regularly varying of index $-\alpha$ with $\alpha\in
[0,1)$. Finally both the asymptotic bounds in \eqref{imrg-rphinew}
follow from Proposition~\ref{proprGiy}.

To complete the proofs of Theorems~\ref{thmerrorequiv}--~\ref{thmerrorequivnew}, we need to check the equivalence of the
statements (\hyperlink{Cauchyremainder}{ii})
and (\hyperlink{Voiculescuremainder}{iii})
involving the remainder terms in Laurent expansion of Cauchy and
Voiculescu transforms for all the theorems and the asymptotic
equivalences between the remainder terms in Laurent series expansion of
Cauchy and Voiculescu transforms and their real and imaginary parts
given in \eqref{rg-rphi}--\eqref{imrg-rphinew}.\vadjust{\goodbreak} Note that all these
claims about Cauchy and Voiculescu transforms of $\mu$ have analogs
about $H_\mu$ and $R_\mu$, due to the facts that $r_G(z)=r_H(1/z)$ and
$r_\phi(z)=r_R(1/z)$. We shall actually deal with the functions $H_\mu$
and $R_\mu$.

For any probability measure $\mu\in\mathcal M_p$, the function
$H\equiv H_\mu$ is invertible, belongs to the class $\mathcal H$ and
the leading term of its Taylor expansion is $z$. Further, by
Proposition A.3 of~\cite{benaych2006taylor}, the above statement about
$H$ is equivalent to the same statement about its inverse, denoted by
$L\equiv L_\mu$. Since the leading term of Taylor expansion of $L$ has
leading term $z$, the leading term of Taylor's expansion of $L(z)/z$ is
$1$, and it is also in $\mathcal H$. Define $K(z)=z/L(z)$. Then $K$ is
also in $\mathcal H$, and its Taylor expansion has leading term $1$. We
shall also use the following facts obtained from~\eqref{eqRphidefn}:
%
%
\begin{equation} \label{eqR-K}
zR_\mu(z) = \bigl(K(z) - 1\bigr) \quad\mbox{and}\quad zK(z) = z \bigl(1+zR_\mu(z)\bigr).
\end{equation}
Hence a Taylor expansion of $K$ will also lead to a Taylor expansion of
$R$ of degree one less than that of $K$. However, due to the definition
of the remainder term of the Taylor expansion given in \eqref{eqdefremainder}, the corresponding remainder terms will be related by
$r_K\equiv r_R$. Thus, we can move from the function $r_H$ to $r_K$
$(\equiv r_R)$ through inverse and reciprocal and vice versa as follows:
\begin{eqnarray} \label{scheme}
H(z)&& \mathop{\hbox to 2cm{\leftarrowfill\hspace*{-2pt}\rightarrowfill}}_{\mathrm{Proposition\ \scriptsize{\ref
{inverse-taylor}}}}^{L(z)=H^{-1}(z)} L(z)=z\cdot\frac{L(z)}z
\nonumber
\\[-8pt]
\\[-8pt]
\nonumber
&&\mathop{\hbox to 2cm{\leftarrowfill\hspace*{-2pt}\rightarrowfill}}_{\mathrm{Proposition\ \scriptsize{\ref
{fraction-taylor}}}}^{K(z)=
{z}/{L(z)}} zK(z) \mathop{\hbox to 2cm{\leftarrowfill\hspace*{-2pt}\rightarrowfill}}_{r_K=r_R}^{R(z)={(K(z)-1)}/z} R(z).
\end{eqnarray}
These observations set up the stage for Propositions \ref
{fraction-taylor} and~\ref{inverse-taylor}. We shall use the class
$\mathcal R_{p,0}$ for Theorems~\ref{thmerrorequiv}--\ref{thmerrorequiv-0} and the class $\mathcal R_{p,\beta}$ with any $\beta\in
(0,1/2)$ for Theorem~\ref{thmerrorequivnew}.

Suppose $\mu\in\mathcal M_p$ with $\alpha\in[p,p+1)$. This condition
holds for Theorems~\ref{thmerrorequiv}--\ref{thmerrorequiv-0},
and we prove these three theorems first. In these cases, $H_\mu(z)$ and
$zK(z) = z(1+zR_\mu(z))$ necessarily have Taylor expansions of the form
given in the hypothesis (\hyperlink{R1}{R1}) for the class $\mathcal R_{p,0}$
with $r_H(z)\ll1$ and $r_R(z)\ll1$ as $z\to\infty$.

For all three theorems, first assume statement (\hyperlink{Cauchyremainder}{ii})
that $\Im r_G(iy)$ is regularly varying of index $-(\alpha-p)$. Then,
from the already proven lower bounds in \eqref{rg-rphi}--\eqref
{rerg-rphi-0}, we have the asymptotic lower bounds for $r_G(z)$, $\Re
r_G(iy)$ and $\Im r_G(iy)$ under the setup of each of the three
theorems. They translate to the asymptotic lower bounds for the
function $H_\mu$, as required by the hypotheses (\hyperlink{R2}{R2}) and (\hyperlink{R3}{R3}).
The asymptotic upper bound in (\hyperlink{R2}{R2}) holds, as the remainder
term in Taylor series expansion of $H$ satisfies $r_H\ll1$. For
Theorem~\ref{thmerrorequiv-0}, we have $p=0$, and we need to check
the extra condition (\hyperlink{R4prime}{R4$'$}), which follows from the already
proven asymptotic equivalences \eqref{imrg-rphi-0} and \eqref
{rerg-rphi-0}. Thus, for each of Theorems~\ref{thmerrorequiv}--\ref
{thmerrorequiv-0}, $H_\mu$ belongs to $\mathcal R_{p,0}$.

We now refer to the schematic diagram given in \eqref{scheme}. As
$H_\mu
$ is also invertible with $L=H^{-1}\in\mathcal H$, by Proposition\vadjust{\goodbreak} \ref
{inverse-taylor}, we also have $L\in\mathcal R_{p,0}$ and $r_H(z) \sim
- r_L(z)$, $\Re r_H(-iy) \sim-\Re r_L(-iy)$ and $\Im r_H(-iy) \sim
-\Im r_L(-iy)$. Clearly, then Proposition~\ref{fraction-taylor} applies
to the function $L(z)/z$, which has reciprocal $K\in\mathcal H$. Thus,
$r_K$ and $r_L$ satisfy the relevant asymptotic equivalences.
Furthermore, since, $r_R \equiv r_K$, combining, we have $r_H(z) \sim
r_R(z)$, $\Re r_H(-iy) \sim\Re r_R(-iy)$ and $\Im r_H(-iy) \sim\Im
r_R(-iy)$. Further, for Theorem~\ref{thmerrorequiv-0}, we have $p=0$
and $H\in\mathcal R_{p,0}$ satisfies (\hyperlink{R4prime}{R4$'$}). Hence, we also
have $\Re r_R(-iy) \approx\Im r_R(-iy)$. Then $R_\mu$ inherits the
appropriate properties from $H_\mu$ and passes them on to $\phi_\mu$,
which gives us the statement (\hyperlink{Voiculescuremainder}{iii}) about the
remainder term in Laurent expansion of Voiculescu transform in each of
Theorems~\ref{thmerrorequiv}--\ref{thmerrorequiv-0}.

Conversely, assume the statement (\hyperlink{Voiculescuremainder}{iii}). Then the
assumptions on $r_\phi$ imply the analogous properties for $r_R\equiv
r_K$. Further, as $\mu$ is in $\mathcal M_p$, $zK(z) = z(1+zR_\mu(z))$
satisfies the hypothesis~(\hyperlink{R1}{R1}) for the class $\mathcal R_{p,0}$.
Also, the remainder term of Taylor series expansion of $zK(z)$ is also
given by $r_R\equiv r_K\ll1$. The lower bound for the imaginary part of
the remainder term in the hypothesis (\hyperlink{R3}{R3}) follows from its
regular variation and the fact that $\alpha\in[p, p+1)$. The lower
bound in the hypothesis~(\hyperlink{R2}{R2}) is part of the statement (\hyperlink{Voiculescuremainder}{iii}). The lower bound for the real part of the
remainder term in the hypothesis (\hyperlink{R3}{R3}) is also a part of the
statement (\hyperlink{Voiculescuremainder}{iii}) for Theorems~\ref{thmerrorequiv} and~\ref{thmerrorequiveqp}, while it follows from the
statement (\hyperlink{Voiculescuremainder-0}{iii}) for Theorem~\ref{thmerrorequiv-0}, as both the real and imaginary parts become asymptotically
equivalent and regularly varying of index $\alpha$ with $\alpha\in
[0,1)$. Finally, the asymptotic equivalence in (\hyperlink{R4prime}{R4$'$}) for
Theorem~\ref{thmerrorequiv-0} is a part of the statement (\hyperlink{Voiculescuremainder-0}{iii}).
Thus, again for each of Theorems~\ref{thmerrorequiv}--\ref{thmerrorequiv-0}, $zK(z)$ belongs to $\mathcal
R_{p,0}$. Then apply Proposition~\ref{fraction-taylor} on $K$ and then
Proposition~\ref{inverse-taylor} on $z/K(z)=L(z)$ to obtain $H_\mu(z)$.
Arguing, by checking the asymptotic equivalences as in the direct case,
we obtain the required conclusions about $r_H$ and hence $r_G$ given in
the statement (\hyperlink{Cauchyremainder}{ii}) for each of Theorems~\ref{thmerrorequiv}--\ref{thmerrorequiv-0}.

The argument is same in the case $\alpha=p+1$, which applies to
Theorem~\ref{thmerrorequivnew}, with the observation that the
stronger bounds required in the hypotheses (\hyperlink{R2}{R2}), (\hyperlink{R3}{R3})
and (\hyperlink{R4}{R4$''$}) with $\beta>0$ are assumed for $r_\phi$ and hence for
$r_R$ and is proved for $r_G$ and hence for $r_H$ in Propositions \ref
{propupperboundrG} and~\ref{proprGiy}.
\end{pf*}

\section*{Acknowledgments}
A part of this work was done while the authors were visiting the
Institute of Mathematical Sciences, Chennai. The authors would like to
thank V. S. Sunder for his hospitality and the stimulating ambience at
the Institute. We thank B. V. Rao and an anonymous referee for reading
an earlier draft and providing recommendations to improve the presentation.

%

%


\printaddresses

\end{document}